\documentclass[amscd,amssymb,11pt]{article}

\normalbaselines
\usepackage{amsmath}
\usepackage{amsthm}
\usepackage{amsfonts}
\usepackage{amssymb}
\usepackage{bbm}
\usepackage{color}
\usepackage{mathrsfs}
\usepackage{wrapfig}
\usepackage{graphicx}
\usepackage{multimedia}
\usepackage{fancyhdr}
\usepackage{titling}
\usepackage[top=1in,bottom=1in,left=1in,right=1in]{geometry}

\definecolor{Red}{rgb}{1.00, 0.00, 0.00}

\definecolor{Blue}{rgb}{0.00, 0.00, 1.00}

\definecolor{Green}{rgb}{0.2,0.5,0.2}

\newtheorem{theorem}{Theorem}
\newtheorem{corollary}[theorem]{Corollary}
\newtheorem{lemma}[theorem]{Lemma}

\newtheorem{assumption}[theorem]{Assumption}
\newtheorem{definition}[theorem]{Definition}

\newtheorem{remark}[theorem]{Remark}

\numberwithin{theorem}{section}
\numberwithin{equation}{section}

\let\Section=\section
\def\section{\setcounter{equation}{0}\Section}\sf
\def\Swiech{{\accent"13S}wie{\hbox{\kern -0.21em\lower 0.79ex\hbox{$\textfont1=\scriptfont1\lhook$}}}ch}
\def\SWIECH{{\accent"13S}WIE{\hbox{\kern -0.26em\lower 0.77ex\hbox{$\textfont1=\scriptfont1\lhook$}}}CH}

\title{\vspace{-3ex}{\bf Stochastic Representations for Solutions to Parabolic Dirichlet Problems for  Nonlocal Bellman Equations}}

\author{
    \textsc{Ruoting Gong}\\
    \textit{Department of Applied Mathematics, Illinois Institute of Technology}\\
    \textit{Chicago, IL 60616, U.S.A.}\\
    \textit{E-mail: rgong2@iit.edu} \\
\vspace{0.1cm}\\
    \textsc{Chenchen Mou}\\
    \textit{Department of Mathematics, UCLA}\\
    \textit{Los Angeles, CA 90095, U.S.A.}\\
    \textit{E-mail: muchenchen@math.ucla.edu}
\vspace{0.1cm}\\
    {\small AND}
\vspace{0.1cm} \\
    \textsc{Andrzej \Swiech}\\
    \textit{School of Mathematics, Georgia Institute of Technology}\\
    \textit{Atlanta, GA 30332, U.S.A.}\\
    \textit{E-mail: swiech@math.gatech.edu}}

\date{}
\begin{document}

\maketitle

\begin{abstract}

We prove a stochastic representation formula for the viscosity solution of Dirichlet terminal-boundary value problem for a degenerate
Hamilton-Jacobi-Bellman integro-partial differential equation in a bounded domain. We show that the unique viscosity solution is the value function of the associated stochastic optimal control problem. We also obtain the dynamic programming principle for the associated stochastic optimal control problem in a bounded domain.

\end{abstract}

\vspace{0.2cm}


\vspace{0.2cm}


\section{Introduction}

Stochastic representation formulas establish natural connections between the study of stochastic processes, and partial differential equations (PDEs) or integro-partial differential equations (integro-PDEs). First formulas of this type appeared in the works of Feynman~\cite{Feynman:1948} and Kac~\cite{Kac:1949}. Since then, the so-called Feynman-Kac formula has been extended and generalized in different directions. Most notably, the dynamic programming principle and the theory of regular and viscosity solutions established rigorous connection between stochastic optimal control problems and fully nonlinear Hamilton-Jacobi-Bellman (HJB) equations, thus providing representations for solutions to HJB equations as value functions of the associated optimal control problems. Such techniques found applications in many areas, such as finance, economics, physics, biology, and engineering. Various results on stochastic representation formulas for regular and viscosity solutions to HJB and Isaacs PDEs in bounded and unbounded domains and their connections to stochastic optimal control problems can be found, for instance, in~\cite{BuckdahnLi:2008,FlemingSoner:2006,FlemingSouganidis:1989,Katsoulakis:1995,Kovats:2009,Krylov:1980,Lions:1983(1),Lions:1983(2),MaZhang:2002,Nisio:2015,
PardouxPeng:1992,Swiech:1996,Touzi:2002,Touzi:2013,YongZhou:1999,Zhang:2005} and the references therein. The literature on the subject is huge.

There are also many existing results for integro-PDEs. Viscosity solutions to HJB integro-PDEs and their stochastic representation formulas as value functions of the associated stochastic optimal control problems were originally investigated in~\cite{Soner:1986,Soner:1988}. Stochastic optimal control of jump-diffusion processes and various results about the associated HJB equations are discussed in~\cite{OksendalSulem:2007}. The presentation in~\cite{OksendalSulem:2007} focuses more on applications, and is not always completely rigorous. A HJB obstacle problem on the whole space $\mathbb{R}^{n}$ associated to the optimal stopping of a controlled jump-diffusion process was studied in~\cite{Pham:1998}, where the value function is proved to be the unique viscosity solution of the obstacle problem. However, the proof of the dynamic programming principle is only sketched there and some other arguments are left without full details. Stochastic optimal control problem in the whole space, which also included control of jump diffusions, was considered in~\cite{ElKarouiNguyeneanblanc:1987} and the dynamic
programming principle was proved there. A special two-dimensional HJB integro-PDE associated to an optimal control problem with jump processes originating from mathematical finance was studied in~\cite{Ishikawa:2004}. The unique viscosity solution was identified as the value function, and the full proof of the dynamic programming principle was presented. The Dirichlet problem for stable-like operators and corresponding stochastic representations were 
studied in~\cite{ArapostathisBiswasCaffarelli:2016}. Stochastic representation formulas based on backward stochastic differential equations for different types of integro-PDEs were obtained in~\cite{BarlesBuckdahnPardoux:1997,BuckdahnHuLi:2011,KharroubiPham:2015,NualartSchoutens:2001}. Value functions of stochastic differential games for jump diffusions, the dynamic programming principle, and the associated Isaacs equations were studied in~\cite{Biswas:2012,BiswasJakobsenKarlsen:2010,BuckdahnHuLi:2011,KoikeSwiech:2013}. There are many other related results. For instance, results for obstacle problems for integro-PDEs related to optimal stopping and impulse control problems with jump-diffusions can be found in~\cite{GarroniMenaldi:2002,Seydel:2009}. A time-inhomogeneous L\'{e}vy model with discontinuous killing rates was investigated in~\cite{Glau:2016}, and stochastic representation formulas for the associated integro-PDE was derived. Feynman-Kac formulas for regime-switching jump-diffusion processes driven by L\'{e}vy processes were also obtained in~\cite{ZhuYinBaran:2015}. A proof of the dynamic programming principle and the representation formula for a viscosity solution to a HJB integro-PDE in an infinite dimensional Hilbert space is contained in~\cite{SwiechZabczyk:2016}.

In this article, we consider a stochastic optimal control problem for a general class of time- and state-dependent controlled stochastic differential equations (SDEs) driven by a general L\'{e}vy process. Our main focus is a stochastic representation formula for the unique viscosity solution to the Dirichlet terminal-boundary value problem for the associated degenerate HJB integro-PDE \eqref{eq:HJB}$-$\eqref{eq:TermBoundCondHJB} in a bounded domain. This is a classical problem which is very technical and whose full details are often omitted or overlooked, especially for problems in bounded domains. We identify the unique viscosity solution to \eqref{eq:HJB}$-$\eqref{eq:TermBoundCondHJB} as the value function of the associated stochastic optimal control problem. We make mild assumptions on the regularity of the domain and the non-degeneracy of the controlled diffusions along the boundary to guarantee the existence of a continuous viscosity solution and thus the continuity of the value function. These assumptions are needed to apply the integro-PDE results of
\cite{Mou1:2016}. Instead of proving directly the continuity of the value function and the dynamic programming principle, we start with the viscosity solution of the HJB integro-PDE, which can be obtained by Perron's method, and show that it must be the value function. Our method is similar to that of~\cite{FlemingSoner:2006}. Similar methods have been also used for Isaacs equations and differential game problems in~\cite{KoikeSwiech:2013,Swiech:1996}. We approximate our HJB integro-PDE by equations which are non-degenerate, have finite control sets, more regular coefficients, and smooth terminal-boundary values, and are considered on slightly enlarged domains. Such equations have classical solutions for which the representation formulas can be obtained. We then pass to the limits with various approximations. The main difficulties come from the fact that we are dealing with a bounded region and hence we need a lot of technical estimates involving the analysis of the behavior of stochastic processes and their exit times. We also need precise knowledge about the behavior of the viscosity solutions of the perturbed equations along the boundaries of their domains, which are obtained by comparison theorems and the constructions of appropriate viscosity subsolutions and supersolutions. We use regularity and existence results from~\cite{Mou:2017,Mou1:2016}.
The approximations using finite control sets, more regular coefficients, and smooth terminal-boundary values are needed to employ a regularity
theorem from~\cite{Mou1:2016}. Enlarged domains are used to handle exit time estimates. We remark that making the HJB integro-PDE non-degenerate by adding a small Laplacian term to the equation corresponds to the introduction of another independent Wiener process on the level of the stochastic control problem, and hence to possible enlargement of the reference probability space. As a byproduct of our method, we obtain the dynamic programming principle for the associated stochastic optimal control problem. Moreover our method provides a fairly explicit way to construct $\varepsilon$-optimal controls using approximating HJB equations.

The paper is organized as follows. Section \ref{sec:preliminary} contains the setup of the problem and some preliminary estimates, which will be needed throughout the paper. Section \ref{sec:DP} establishes the stochastic representation formula and the dynamic programming principle for the solution to the HJB integro-PDE terminal boundary value problem \eqref{eq:HJB}$-$\eqref{eq:TermBoundCondHJB}. The results are first obtained in Subsection \ref{Subsec:DPSmoothSols} for the classical solution to \eqref{eq:HJB}$-$\eqref{eq:TermBoundCondHJB}, and are then extended, in Subsection \ref{Subsec:DPVisSolFinCtrl}, to the viscosity solution to \eqref{eq:HJB}$-$\eqref{eq:TermBoundCondHJB} with a finite control set. Using various approximation arguments, in Subsection \ref{Subsec:DPVisSolGenCtrl}, the representation formula and the dynamic programming principle is finally proved for the viscosity solution to \eqref{eq:HJB}$-$\eqref{eq:TermBoundCondHJB} when the control set is a general Polish space. Section \ref{sec:VisConstrn} provides a construction of viscosity sub/supersolutions to \eqref{eq:HJB}$-$\eqref{eq:TermBoundCondHJB}.

\section{Preliminaries}\label{sec:preliminary}

\subsection{Setup and Assumptions}\label{subsec:Model}

Throughout this article, let $T>0$ be a fixed terminal time, let $t\in[0,T)$ be an arbitrary fixed time, and let $d,m_{1},m_{2}\in\mathbb{N}$ be fixed positive integers. Let $O\subset\mathbb{R}^{d}$ be a bounded domain. Let $\mathcal{U}$ be a Polish space equipped with its metric $d_{\mathcal{U}}$. Let $Q:=[0,T)\times O$, and let $Q^{0}:=[0,T)\times\mathbb{R}^{d}$. Let $\nu$ be a L\'{e}vy measure, i.e., a measure on $\mathcal{B}(\mathbb{R}^{m_{2}}_{0})$, where $\mathbb{R}^{m_{2}}_{0}:=\mathbb{R}^{m_{2}}\setminus\{0\}$, such that
\begin{align*}
\int_{\mathbb{R}^{m_{2}}_{0}}\left(|z|^{2}\wedge 1\right)\nu(dz)<+\infty.
\end{align*}

We say that $\mu:=(\Omega,\mathscr{F},\mathscr{F}_{s}^{t},\mathbb{P},\mathcal{W},\mathcal{L})$ is a \emph{generalized reference probability space} if it satisfies the following conditions:
\begin{itemize}
\item $(\Omega,\mathscr{F},\mathbb{P})$ is a complete probability space, and $\{\mathscr{F}_{s}^t\}_{s\in [t,T]}$ is a filtration of sub-$\sigma$-fields of $\mathscr{F}$ satisfying the usual conditions;
\item $\mathcal{W}$ is an $m_{1}$-dimensional standard $\mathscr{F}_{s}^{t}$-Brownian motion on $(\Omega,\mathscr{F},\mathbb{P})$;
\item $\mathcal{L}$ is an $m_{2}$-dimensional $\nu$ $\mathscr{F}_{s}^{t}$-L\'{e}vy process on $(\Omega,\mathscr{F},\mathbb{P})$. That is, $\mathcal{L}$ is an $\mathcal{F}^{t}_{s}$-adapted stochastic process with $\mathbb{P}$-a.$\,$s. c\'{a}dl\'{a}g trajectories, such that for all $t\leq t_1\leq t_2\leq T$, the random variable $\mathcal{L}(t_{2})-\mathcal{L}(t_{1})$ is independent of $\mathscr{F}^{t}_{t_{1}}$, and
    \begin{align*}
    \mathbb{E}\left(e^{i\left(\mathcal{L}(t_{2})-\mathcal{L}(t_{1})\right)\cdot z}\right)=e^{-(t_{2}-t_{1})\psi(z)},\quad z\in\mathbb{R},
    \end{align*}
    where
    \begin{align*}
    \psi(z)=\int_{\mathbb{R}^{m_{2}}_{0}}\left(1-e^{iz\cdot y}+\mathbbm{1}_{\{|y|<1\}}iz\cdot y\right)\nu(dy).
    \end{align*}
\end{itemize}
Note that we do not assume here that $\mathcal{W}(t)=0$ or $\mathcal{L}(t)=0$. The jump measure of $\mathcal{L}$ (defined on $\mathcal{B}([t,T])\otimes\mathcal{B}(\mathbb{R}^{m_{2}}_{0})$) is denoted by $N(ds,dz)$, with its compensated measure $\widetilde{N}(ds,dz):=N(ds,dz)-ds\,\nu(dz)$. For more details on L\'{e}vy processes, we refer the reader to~\cite{Appelbaum:2009,Bertoin:1996,PeszatZabczyk:2007}. Finally, we denote by $\mathcal{A}_{\mu}$ the set of all $\mathscr{F}_s^t$-predictable $\mathcal{U}$-valued processes on $[t,T]$, and let $\mathcal{A}_{t}:=\cup_{\mu}\mathcal{A}_{\mu}$, where the union is taken over all generalized reference probability spaces $\mu$ on $[t,T]$.

For any generalized reference probability space $\mu=(\Omega,\mathscr{F},\mathscr{F}_{s}^{t},\mathbb{P},\mathcal{W},\mathcal{L})$, any control $U\in\mathcal{A}_{\mu}$, and any $x\in\mathbb{R}^{d}$, consider an $\mathbb{R}^{d}$-valued stochastic process $X(s;t,x)$, governed by the following controlled SDE:
\begin{align}\label{eq:MainSDEs}
X(s;t,x)&=x+\int_{t}^{s}b\left(r,X(r;t,x),U(r)\right)dr+\int_{t}^{s}\sigma\!\left(r,X(r;t,x),U(r)\right)d\mathcal{W}(r)\nonumber\\
&\quad +\int_{t}^{s}\int_{\mathbb{R}^{m_{2}}_{0}}\gamma\left(r,X(r-;t,x),U(r),z\right)\widetilde{N}(dr,dz),\quad s\in[t,T].
\end{align}
\begin{assumption}\label{assump:SDECoefs}
Throughout this article, we make the following assumptions on the coefficients in the SDE system \eqref{eq:MainSDEs}.
\begin{itemize}
\item [(i)] $\gamma:\overline{Q^{0}}\times\mathcal{U}\times\mathbb{R}^{m_{2}}\rightarrow\mathbb{R}^{d}$ is a measurable function.
\item [(ii)] $b:\overline{Q^{0}}\times\mathcal{U}\rightarrow\mathbb{R}^{d}$ and $\sigma:\overline{Q^{0}}\times\mathcal{U}\rightarrow\mathbb{R}^{d\times m_{1}}$ are uniformly continuous functions.
\item [(iii)] There exist a universal constant $C>0$, a modulus of continuity $\varpi:\mathbb{R}^{+}\rightarrow\mathbb{R}^{+}$ with $\lim_{r\rightarrow 0}\varpi(r)=\varpi(0)=0$, and a Borel measurable function $\rho:\mathbb{R}^{m_{2}}_{0}\rightarrow\mathbb{R}$ which is bounded on any bounded subset of $\mathbb{R}^{m_{2}}_{0}$, and which satisfies
    \begin{align*}
    \text{(a)}\quad\inf_{|z|>r}\rho(z)>0,\quad\forall\,\,r>0;\qquad\text{(b)}\quad M:=\int_{\mathbb{R}^{m_{2}}_{0}}\rho^{2}(z)\,\nu(dz)<\infty,
    \end{align*}
    such that for any $u,u_{1},u_{2}\in\mathcal{U}$, $s\in[0,T)$, $(s_{1},y_{1}),(s_{2},y_{2})\in \overline{Q^{0}}$, and any $z\in\mathbb{R}^{m_{2}}_{0}$,
    \begin{align}
    \left|b(s,y_{1},u)-b(s,y_{2},u)\right|+\left\|\sigma(s,y_{1},u)-\sigma(s,y_{2},u)\right\|\leq C\left|y_{1}-y_{2}\right|,\qquad\quad\,\,\nonumber\\
    \left|\gamma(s_{1},y_{1},u_{1},z)-\gamma(s_{2},y_{2},u_{2},z)\right |\leq C\rho(z)\left(\varpi\!\left(d_{\mathcal{U}}(u_{1},u_{2})+\left|s_{1}-s_{2}\right|\right)+\left|y_{1}-y_{2}\right|\right),\nonumber\\
    \label{eq:inftybounds} \left\||b|\right\|_{L^{\infty}(\overline{Q^{0}}\times\mathcal{U})}+\left\|\|\sigma\|\right\|_{L^{\infty}(\overline{Q^{0}}\times\mathcal{U})}\leq C,\quad\left\||\gamma(\cdot,\cdot,\cdot,z)|\right\|_{L^{\infty}(\overline{Q^{0}}\times\mathcal{U})}\leq C\rho(z).\quad\,\,\,\,
    \end{align}
\end{itemize}
\end{assumption}

To avoid cumbersome notation, from now on, we will be writing $\|b\|_{L^{\infty}(\overline{Q^{0}}\times\mathcal{U})}$, $\|\sigma\|_{L^{\infty}(\overline{Q^{0}}\times\mathcal{U})}$, etc., for $\||b|\|_{L^{\infty}(\overline{Q^{0}}\times\mathcal{U})}$, $\|\|\sigma\|\|_{L^{\infty}(\overline{Q^{0}}\times\mathcal{U})}$, etc., i.e., we will be omitting the inside norms in the notation for the supremum norms of vector and matrix valued functions.

For any fixed $t\in[0,T)$ and any generalized reference probability space $\mu$, let $\mathcal{H}$ be the space of all $\mathscr{F}_{s}^{t}$-adapted c\`{a}dl\`{a}g processes $Y$ such that
\begin{align*}
\|Y\|_{\mathcal{H}}:=\left(\mathbb{E}\left(\sup_{s\in[t,T]}|Y(s)|^{2}\right)\right)^{1/2}<\infty.
\end{align*}
The next result provides the existence of a unique strong c\`{a}dl\`{a}g solution to \eqref{eq:MainSDEs}. The proof is quite standard and is thus omitted here. The reader is referred to, e.g., \cite[Section 6.2]{Appelbaum:2009},~\cite[Section 1.3]{OksendalSulem:2007} or \cite{SwiechZabczyk:2016}, where the proof of existence in the case of a similar controlled SDE in a Hilbert space is provided.
\begin{theorem}\label{thm:StrSolsSDEs}
Let Assumption \ref{assump:SDECoefs} be satisfied and let $U\in\mathcal{A}_{\mu}$. Then, for every $x\in\mathbb{R}^{d}$, the SDE \eqref{eq:MainSDEs} admits a unique strong solution $X(s;t,x)$ in the space $\mathcal{H}$. Moreover, there exists a constant $K_{1}=K_{1}(C,T,M)>0$, depending on $C$, $T$, and $M$, such that
\begin{align}\label{eq:Est2ndMoment}
\mathbb{E}\left(\sup_{s\in[t,T]}\left|X(s;t,x)\right|^{2}\right)\leq K_{1}\left(1+|x|^{2}\right).
\end{align}
\end{theorem}

We notice that, when $x\in O$, which is a bounded subset of $\mathbb{R}^{d}$, the right-hand side of \eqref{eq:Est2ndMoment} is bounded by a constant independent of $x$. For any $(t,x)\in Q^{0}$, let
\begin{align*}
\tau(t,x):=\inf\left\{s\in[t,T]:\,\left(s,X(s;t,x)\right)\not\in Q\right\},
\end{align*}
with the convention that $\inf\emptyset=T$. Throughout this article, when there is no confusion of initial condition, we will skip the initial data $(t,x)$ in the expressions of SDE solutions and exit times, and write $X(s)$ and $\tau$ instead of $X(s;t,x)$ and $\tau(t,x)$, respectively. Clearly,
\begin{align*}
\left(\tau,X(\tau)\right)\in\partial_{\text{np}}Q:=\left([0,T)\times O^{c}\right)\cup\left(\{T\}\times\mathbb{R}^{d}\right).
\end{align*}
Define the cost functional as
\begin{align}\label{eq:CostFunt}
J_{\mu}\left(t,x;U\right):=\mathbb{E}\left(\int_{t}^{\tau(t,x)}\Gamma\left(s,X(s;t,x),U(s)\right)ds+\Psi\left(\tau(t,x),X(\tau(t,x);t,x)\right)\right).
\end{align}
\begin{assumption}\label{assump:GammaPsi}
Throughout this article, we make the following assumptions on $\Gamma$ and $\Psi$.
\begin{itemize}
\item [(i)] $\Psi:\overline{Q^{0}}\rightarrow\mathbb{R}$ is a bounded continuous function.
\item [(ii)] $\Gamma:\overline{Q^{0}}\times\mathcal{U}\rightarrow\mathbb{R}$ is a bounded uniformly continuous function.
\end{itemize}
\end{assumption}

We will consider the stochastic control problem by first taking the infimum of the cost functional \eqref{eq:CostFunt} over all $U\in\mathcal{A}_{\mu}$, i.e.,
\begin{align}\label{eq:ValueFunt1}
V_{\mu}(t,x):=\inf_{U\in\mathcal{A}_{\mu}}J_{\mu}\left(t,x;U\right),
\end{align}
and then by taking the infimum of \eqref{eq:ValueFunt1} over all generalized reference probability spaces, i.e.,
\begin{align*}
V(t,x):=\inf_{\mu}V_{\mu}(t,x)=\inf_{U\in\mathcal{A}_t}J_{\mu}\left(t,x;U\right).
\end{align*}
The corresponding HJB equation is then given by
\begin{align}\label{eq:HJB}
\inf_{u\in\mathcal{U}}\left(\mathscr{A}^{u}W(t,x)+\Gamma(t,x,u)\right)=0\quad\text{in }\,Q,
\end{align}
with terminal-boundary condition
\begin{align}\label{eq:TermBoundCondHJB}
W(t,x)=\Psi(t,x),\quad (t,x)\in\partial_{\text{np}}Q,
\end{align}
where
\begin{align}
\mathscr{A}^{u}W(t,x)&:=\frac{\partial W}{\partial t}(t,x)+\frac{1}{2}\text{tr}\left(a(t,x,u)D_{x}^{2}W(t,x)\right)+b(t,x,u)\cdot D_{x}W(t,x)\nonumber\\
\label{eq:GenX} &\quad\,\,+\int_{\mathbb{R}^{m_{2}}_{0}}\left(W\!\left(t,x\!+\!\gamma(t,x,u,z)\right)-W(t,x)-D_{x}W(t,x)\cdot\gamma(t,x,u,z)\right)\nu(dz),
\end{align}
and where $a(t,x,u):=\sigma(t,x,u)\sigma^T(t,x,u)$.

We now introduce the assumptions about the bounded domain $O$. Since $O$ is bounded, Assumptions \ref{assump:SDECoefs} and \ref{assump:GammaPsi} are quite standard. However to apply the integro-PDE results of
\cite{Mou1:2016}, we will need extra conditions on $O$ and $\sigma$, Assumptions \ref{assump:Domain O}
and \ref{assump:ellpiticity along boundary}. More precisely, these two assumptions allow to construct viscosity sub/supersolutions 
of various approximating equations having uniform modulus of continuity in Section \ref{sec:VisConstrn}. They are
used to apply the existence and regularity results for solutions to our non-local HJB equations (see Theorem \ref{thm:uer} and the discussion after the proof of Theorem \ref{thm:wdeltan}) and to obtain uniform convergence of solutions of various approximating equations (see Lemmas \ref{lem:ConvWdeltanWdelta}, \ref{lem:ConvOverlineWdeltanWdelta} and \ref{lem:Comparison}).

For a set $\widetilde{O}\subset\mathbb{R}^{d}$, we define the proximal normal cone to $\widetilde{O}$ at $x\in\partial\widetilde{O}$ by
\begin{align*}
N(\widetilde{O},x):=\left\{n\in\mathbb{R}^{d}:\,\text{there exists $\ell>0$, such that }x\in P(\widetilde{O},x+\ell n)\right\},
\end{align*}
where
\begin{align*}
P(\widetilde{O},y)=\left\{z\in\partial\widetilde{O}:\,\inf_{p\in\widetilde{O}}|p-y|=|z-y|\right\}.
\end{align*}
The set $\widetilde{O}$ is said to be \emph{$\eta$-prox-regular} for some $\eta>0$ if, for any $x\in\partial\widetilde{O}$ and any unit vector $n\in N(\widetilde{O},x)$, we have
\begin{align*}
B_{\eta}(x+\eta n)\cap\widetilde{O}=\emptyset,
\end{align*}
where, and hereafter, $B_{r}(y)$ (respectively, $\overline{B}_{r}(y)$) denotes the open (respectively, closed) ball in $\mathbb{R}^{d}$ centered at $y$ with radius $r>0$. We refer the reader to, e.g., \cite{Bounkhel:2012,ClarkeLedyaevSternWolenski:1998,ClarkeSternWolenski:1995,PoliquinRockafellarThibault:2000}, for the properties of $\eta$-prox-regular sets.
\begin{figure}
\centering
\includegraphics[width=8.5cm]{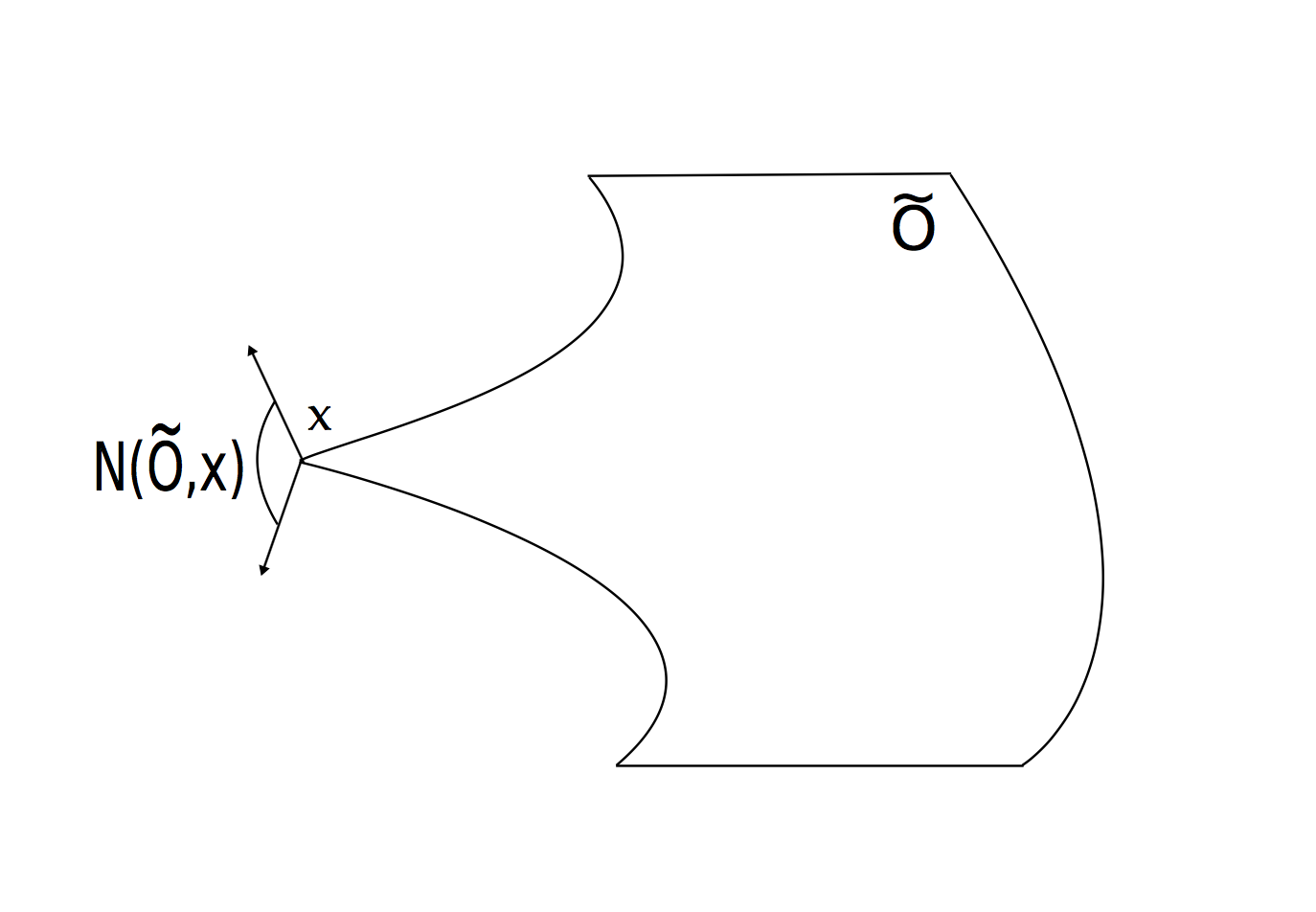}\\
\caption{\scriptsize{An $\eta$-prox-regular set}}
\end{figure}

\begin{assumption}\label{assump:Domain O}
Throughout this article, we assume that $O$ is $\eta$-prox-regular, for some fixed $0<\eta<1$.
\end{assumption}
\begin{theorem}\label{thm:proxdomain}
If $O$ is $\eta$-prox-regular for some $\eta>0$, then, for any $0<\delta<\eta/2$, the set $O_{\delta}:=\{y\in\mathbb{R}^{d}:\,\text{dist}\,(y,O)<\delta\}$ satisfies the uniform exterior ball condition with a uniform radius $\eta/2$, i.e., for any $x\in\partial O_{\delta}$, there exists $y_{x}\in O_{\delta}^{c}$, such that $\overline{B}_{\eta/2}(y_{x})\cap\overline{O}_{\delta}=\{x\}$.
\end{theorem}
\begin{proof}
For any $x\in\partial O_{\delta}$, there exists $\tilde{x}\in\partial O$ such that $|x-\tilde{x}|=\delta$, and thus $x-\tilde{x}\in N(O,\tilde{x})$. Since $O$ is $\eta$-prox-regular, we have $\tilde{x}\in\overline{B}_{\eta}(\tilde{x}+\eta(x-\tilde{x})/\delta)\cap\overline{O}$ and $\overline{B}_{\eta}(\tilde{x}+\eta(x-\tilde{x})/\delta)\cap O=\emptyset$. It is then easy to see that
\begin{align*}
\overline{B}_{\eta/2}\left(\tilde{x}+\left(\frac{\eta}{2}+\delta\right)\frac{x-\tilde x}{\delta}\right)\cap\overline{O}_{\delta}=\{x\},
\end{align*}
which completes the proof.
\end{proof}

Throughout this paper, we make the following parabolicity assumption along the boundary $\partial O$.
\begin{assumption}\label{assump:ellpiticity along boundary}
There exists a constant $\lambda>0$ such that, for any $x\in\partial O$ and $n_{x}\in N(O,x)$,
\begin{align*}
n_{x}\,\sigma(t,x,u)\,\sigma^{T}(t,x,u)\,n_{x}^{T}\geq\lambda,\quad\text{for any }\,t\in [0,T)\,\,\,\text{and}\,\,\,u\in\mathcal{U}.
\end{align*}
\end{assumption}

Throughout the paper, we will use the following function spaces on cylindrical regions $\mathcal Q=[a,b)\times\mathcal{O}$, where $a<b$ and $\mathcal{O}$ is an open subset of $\mathbb{R}^{d}$. ${\rm USC}(\mathcal{Q})$ (respectively, ${\rm LSC}(\mathcal{Q})$) is the space of upper (respectively, lower) semi-continuous functions on $\mathcal{Q}$. $C(\mathcal{Q})$ (respectively, $C(\overline{\mathcal{Q}})$) is the space of continuous functions on $\mathcal{Q}$ (respectively, $\overline{\mathcal{Q}}$). ${\rm Lip}(\mathcal{Q})$ denotes the space of Lipschitz continuous functions on $\mathcal{Q}$. $C^{1,2}(\mathcal{Q})$ is the space of functions $\varphi:\mathcal{Q}\rightarrow\mathbb{R}$ such that $\varphi$, $\varphi_{t}$, $D_{x}\varphi$, and $D_{x}^{2}\varphi$ are continuous on $\mathcal{Q}$. $C^{1,2}(\overline{\mathcal{Q}})$ is the space of functions $\varphi\in C^{1,2}(\mathcal{Q})$ such that $\varphi$, $\varphi_{t}$, $D_{x}\varphi$, and $D_{x}^{2}\varphi$ extend continuously to $\overline{\mathcal{Q}}$. $C^{1+\alpha/2,\,2+\alpha}(\mathcal{Q})$, where $\alpha\in(0,1)$, is the space of functions $\varphi\in C^{1,2}(\mathcal{Q})$ such that
\begin{align}\label{c2alpha}
&\left\|\varphi\right\|_{L^\infty(\mathcal{Q})}+\left\|\varphi_{t}\right\|_{L^\infty(\mathcal{Q})}+\left\|D_{x}\varphi\right\|_{L^\infty(\mathcal{Q})}+\left\|D_{x}^{2}\varphi\right\|_{L^\infty(\mathcal{Q})}\\
&\!+\!\sup_{\substack{(t,x),(s,y)\in\mathcal{Q} \\ (t,x)\neq(s,y)}}\!\!\frac{\left|\varphi(t,x)\!-\!\varphi(s,y)\!-\!\varphi_{t}(s,y)(t\!-\!s)\!-\!D_{x}\varphi(s,y)\!\cdot\!(x\!-\!y)\!-\!(x\!-\!y)^{T}\!D_{x}^{2}\varphi(s,y)(x\!-\!y)/2\right|}{\left(|t-s|+|x-y|^{2}\right)^{1+\alpha/2}}<\infty.
\nonumber
\end{align}
If $\varphi\in C^{1+\alpha/2,\,2+\alpha}(Q^0)$ then $\varphi$, $\varphi_{t}$, $D_{x}\varphi$, $D_{x}^{2}\varphi$ extend continuously to 
$\overline{Q^0}$ and \eqref{c2alpha} is satisfied with $\mathcal{Q}=Q^0$ replaced by $\mathcal{Q}=\overline{Q^0}$. To emphasize that functions in $C^{1+\alpha/2,\,2+\alpha}(Q^0)$ are defined on $\overline{Q^0}$, we will denote this space by 
$C^{1+\alpha/2,\,2+\alpha}(\overline{Q^0})$.
$C^{1+\alpha/2,\,2+\alpha}_{\rm loc}(\mathcal{Q})$ is the space of functions $\varphi:\mathcal{Q}\rightarrow\mathbb{R}$ such that, $\varphi\in
C^{1+\alpha/2,\,2+\alpha}(\widetilde{\mathcal{Q}})$ for any cylindrical region $\widetilde{\mathcal{Q}}\subset\subset\mathcal{Q}$. Finally, the space ${\rm USC}_{b}(\mathcal{Q})$ (respectively, ${\rm LSC}_{b}(\mathcal{Q})$, $C_{b}(\mathcal{Q})$, $C_{b}(\overline{\mathcal{Q}})$,
${\rm Lip}_{b}(\mathcal{Q})$, $C_{b}^{1,2}(\mathcal{Q})$, $C_{b}^{1,2}(\overline{\mathcal{Q}})$) consists of functions in ${\rm USC}(\mathcal{Q})$ (respectively, ${\rm LSC}(\mathcal{Q})$, $C(\mathcal{Q})$, $C(\overline{\mathcal{Q}})$, ${\rm Lip}(\mathcal{Q})$, $C^{1,2}(\mathcal{Q})$, $C^{1,2}(\overline{\mathcal{Q}})$) which are bounded on their respective domains.

To conclude this subsection, we recall the definition of a viscosity solution to \eqref{eq:HJB}.
\begin{definition}\label{def:VisSol}
A function $u\in {\rm USC}_{b}(Q^{0})$ is a viscosity subsolution to \eqref{eq:HJB} if, whenever $u-\varphi$ has a maximum over $Q^{0}$ at $(t,x)\in Q$ for some test function $\varphi\in C_{b}^{1,2}(Q^{0})$,
\begin{align*}
\inf_{u\in\mathcal{U}}\left(\mathscr{A}^{u}\varphi(t,x)+\Gamma(t,x,u)\right)\geq 0.
\end{align*}
A function $u\in {\rm LSC}_b(Q^{0})$ is a viscosity supersolution to \eqref{eq:HJB} if, whenever $u-\varphi$ has a minimum over $Q^{0}$ at $(t,x)\in Q$ for some test function $\varphi\in C_{b}^{1,2}(Q^{0})$,
\begin{align*}
\inf_{u\in\mathcal{U}}\left(\mathscr{A}^{u}\varphi(t,x)+\Gamma(t,x,u)\right)\leq0.
\end{align*}
A function $u\in C_{b}(Q^{0})$ is a viscosity solution to \eqref{eq:HJB} if it is both a viscosity subsolution and a viscosity supersolution to \eqref{eq:HJB}.
\end{definition}

\subsection{Preliminary Estimates}\label{subsec:PreEst}

In this subsection, we prove various estimates for strong solutions to \eqref{eq:MainSDEs}. The proofs of these results follow
rather standard lines of arguments however, since we could not find exact references, we equip them with short proofs for completeness and for the reader's convenience.
\begin{lemma}\label{lem:2MomentEst}
Let Assumption \ref{assump:SDECoefs} be valid. For any $(t,x)\in Q^{0}$, any generalized reference probability space $\mu=(\Omega,\mathscr{F},\mathscr{F}_{s}^{t},\mathbb{P},\mathcal{W},\mathcal{L})$, and any $U\in\mathcal{A}_{\mu}$, let $X(s;t,x)$ be the unique strong c\`{a}dl\`{a}g solution to \eqref{eq:MainSDEs}. Then, for any $t\leq\ell_{1}<\ell_{2}\leq T$,
\begin{align*}
\mathbb{E}\left(\sup_{s\in[\ell_{1},\ell_{2}]}\left|X(s;t,x)-X(\ell_{1};t,x)\right|^{2}\right)\leq K_{2}\left(\ell_{2}-\ell_{1}\right),
\end{align*}
where $K_{2}=K_{2}(C,T,M)>0$ is a constant depending on $C$, $T$, and $M$.
\end{lemma}
\begin{proof}
By Assumption \ref{assump:SDECoefs}, Cauchy-Schwarz and Burkholder-Davis-Gundy inequalities, there exists a universal constant $\Lambda_{1}>0$, such that, denoting $X(s)=X(s;t,x)$,
\begin{align*}
&\mathbb{E}\left(\sup_{s\in[\ell_{1},\ell_{2}]}\left|X(s)-X(\ell_{1})\right|^{2}\right)\\
&\quad\leq 3\,\mathbb{E}\left(\sup_{s\in[\ell_{1},\ell_{2}]}\left|\int_{\ell_{1}}^{s}b\left(r,X(r),U(r)\right)dr\right|^{2}\right)+3\,\mathbb{E}\left(\sup_{s\in[\ell_{1},\ell_{2}]}\!\left|\int_{\ell_{1}}^{s}\sigma\left(r,X(r),U(r)\right)d\mathcal{W}(r)\right|^{2}\right)\\
&\qquad +3\,\mathbb{E}\left(\sup_{s\in[\ell_{1},\ell_{2}]}\left|\int_{\ell_{1}}^{s}\int_{\mathbb{R}^{m_{2}}_{0}}\gamma\left(r,X(r-),U(r),z\right)\widetilde{N}(dr,dz)\right|^{2}\right)\\
&\quad\leq 3T\,\mathbb{E}\left(\int_{\ell_{1}}^{\ell_{2}}\left|b\left(r,X(r),U(r)\right)\right|^{2}dr\right)+3\Lambda_{1}\,\mathbb{E}\left(\int_{\ell_{1}}^{\ell_{2}}\left\|\sigma\left(r,X(r),U(r)\right)\right\|^{2}dr\right)
\end{align*}
\begin{align*}
&\qquad +3\Lambda_{1}\,\mathbb{E}\left(\int_{\ell_{1}}^{\ell_{2}}\int_{\mathbb{R}_{0}^{m_{2}}}\left|\gamma\left(r,X(r-),U(r),z\right)\right|^{2}\nu(dz)\,dr\right)\\
&\quad\leq 3C^{2}\left(T+\Lambda_{1}+\Lambda_{1}M\right)(\ell_{2}-\ell_{1}).
\end{align*}
Letting $K_{2}:=3C^{2}(T+\Lambda_{1}+\Lambda_{1}M)$ completes the proof of the lemma.
\end{proof}

Let $\tilde{b}:\overline{Q^{0}}\times\mathcal{U}\rightarrow\mathbb{R}^{d}$, $\widetilde{\sigma}:\overline{Q^{0}}\times\mathcal{U}\rightarrow\mathbb{R}^{d\times m_{1}}$, and $\widetilde{\gamma}:\overline{Q^{0}}\times\mathcal{U}\times\mathbb{R}^{m_{2}}\rightarrow\mathbb{R}^{d}$. For any generalized reference probability space $\mu=(\Omega,\mathscr{F},\mathscr{F}_{s}^{t},\mathbb{P},\mathcal{W},\mathcal{L})$, any control process $U\in\mathcal{A}_{\mu}$, and any $x\in\mathbb{R}^{d}$, consider another controlled SDE
\begin{align}
\widetilde{X}(s;t,x)&=x+\int_{t}^{s}\tilde{b}\left(r,\widetilde{X}(r;t,x),U(r)\right)dr+\int_{t}^{s}\widetilde{\sigma}\left(r,\widetilde{X}(r;t,x),U(r)\right)d\mathcal{W}(r)\nonumber\\
\label{eq:CompareSDEs} &\quad +\int_{t}^{s}\int_{\mathbb{R}^{m_{2}}_{0}}\widetilde{\gamma}\left(r,\widetilde{X}(r-;t,x),U(r),z\right)\widetilde{N}(dr,dz),\quad s\in[t,T].
\end{align}
When the coefficient functions $\tilde{b}$, $\widetilde{\sigma}$, and $\widetilde{\gamma}$ satisfy Assumption \ref{assump:SDECoefs}, for any $\mu=(\Omega,\mathscr{F},\mathscr{F}_{s}^{t},\mathbb{P},\mathcal{W},\mathcal{L})$, $U\in\mathcal{A}_{\mu}$, and $x\in\mathbb{R}^{d}$, Theorem \ref{thm:StrSolsSDEs} ensures that there exists a unique strong c\`{a}dl\`{a}g solution $\widetilde{X}(s;t,x)$ to \eqref{eq:CompareSDEs}. For any $(t,x)\in Q^{0}$, let
\begin{align*}
\widetilde{\tau}(t,x):=\inf\left\{s\in[t,T]:\,\left(s,\widetilde{X}(s;t,x)\right)\not\in Q\right\}.
\end{align*}
\begin{lemma}\label{lem:CompStrSols}
Let the coefficient functions $\tilde{b}$, $\widetilde{\sigma}$, and $\widetilde{\gamma}$ satisfy Assumption \ref{assump:SDECoefs}. For any $(t,x)\in Q^{0}$, any generalized reference probability space $\mu=(\Omega,\mathscr{F},\mathscr{F}_{s}^{t},\mathbb{P},\mathcal{W},\mathcal{L})$, and any $U\in\mathcal{A}_{\mu}$, let $\widetilde{X}(s;t,x)$ (respectively, $X(s;t,x)$) be the unique strong c\`{a}dl\`{a}g solution to \eqref{eq:CompareSDEs} (respectively, \eqref{eq:MainSDEs}). Then, there exist a constant $K_{3}=K_{3}(C,T,M)>0$, depending only on $C$, $T$, and $M$, such that
\begin{align*}
\mathbb{E}\left(\sup_{s\in[t,T]}\left|X(s;t,x)-\widetilde{X}(s;t,x)\right|^{2}\right)&\leq K_{3}\left(\big\|b-\tilde{b}\big\|_{L^{\infty}(\overline{Q^{0}}\times\mathcal{U})}^{2}+\left\|\sigma-\widetilde{\sigma}\right\|_{L^{\infty}(\overline{Q^{0}}\times\mathcal{U})}^{2}\right)\\
&\quad +K_{3}\int_{\mathbb{R}^{m_{2}}_{0}}\left\|\gamma(\cdot,\cdot,\cdot,z)-\widetilde{\gamma}(\cdot,\cdot,\cdot,z)\right\|_{L^{\infty}(\overline{Q^{0}}\times\mathcal{U})}^{2}\nu(dz).
\end{align*}
\end{lemma}
\begin{proof}
We denote $X(s)=X(s;t,x),\widetilde{X}(s)=\widetilde{X}(s;t,x)$.
By Cauchy Schwarz and Burkholder-Davis-Gundy inequalities, there exists a universal constant $\Lambda_{2}>0$, such that for any $\ell\in[t,T]$,
\begin{align*}
&\mathbb{E}\left(\sup_{s\in[t,\ell]}\left|X(s)-\widetilde{X}(s)\right|^{2}\right)\\
&\quad\leq 6T\,\mathbb{E}\left(\int_{t}^{\ell}\left|b\left(r,X(r),U(r)\right)-b\left(r,\widetilde{X}(r),U(r)\right)\right|^{2}dr\right)\\
&\qquad +6\Lambda_{2}\,\mathbb{E}\left(\int_{t}^{\ell}\left\|\sigma\left(r,X(r),U(r)\right)-\sigma\left(r,\widetilde{X}(r),U(r)\right)\right\|^{2}dr\right)\\
&\qquad +6\Lambda_{2}\,\mathbb{E}\left(\int_{t}^{\ell}\int_{\mathbb{R}^{m_{2}}_{0}}\left|\gamma\left(r,X(r-),U(r),z\right)-\gamma\left(r,\widetilde{X}(r-),U(r),z\right)\right|^{2}\nu(dz)dr\right)\\
&\qquad +6T^{2}\big\|b-\tilde{b}\big\|_{L^{\infty}(\overline{Q^{0}}\times\mathcal{U})}^{2}+6\Lambda_{2}T\left\|\sigma-\tilde{\sigma}\right\|_{L^{\infty}(\overline{Q^{0}}\times\mathcal{U})}^{2}\\
&\qquad +6\Lambda_{2}T\int_{\mathbb{R}^{m_{2}}_{0}}\left\|\gamma(\cdot,\cdot,\cdot,z)-\widetilde{\gamma}(\cdot,\cdot,\cdot,z)\right\|_{L^{\infty}(\overline{Q^{0}}\times\mathcal{U})}^{2}\nu(dz).
\end{align*}
Therefore, by Assumption \ref{assump:SDECoefs}-(iii),
\begin{align*}
\mathbb{E}\left(\sup_{s\in[t,\ell]}\left|X(s)-\widetilde{X}(s)\right|^{2}\right)&\leq 6\big(T+\Lambda_{2}\big)C^{2}(1+M)\int_{t}^{\ell}\mathbb{E}\left(\sup_{s\in[t,r]}\left|X(s)-\widetilde{X}(s)\right|^{2}\right)dr\\
&\quad +6T(T+\Lambda_{2})\left(\big\|b-\tilde{b}\big\|_{L^{\infty}(\overline{Q^{0}}\times\mathcal{U})}^{2}+\left\|\sigma-\tilde{\sigma}\right\|_{L^{\infty}(\overline{Q^{0}}\times\mathcal{U})}^{2}\right)\\
&\quad +6T(T+\Lambda_{2})\int_{\mathbb{R}^{m_{2}}_{0}}\left\|\gamma(\cdot,\cdot,\cdot,z)-\widetilde{\gamma}(\cdot,\cdot,\cdot,z)\right\|_{L^{\infty}(\overline{Q^{0}}\times\mathcal{U})}^{2}\nu(dz).
\end{align*}
The result follows from Gronwall's inequality with $K_{3}:=6T(T+\Lambda_{2})e^{6C^{2}(1+M)(T+\Lambda_{3})T}$.
\end{proof}

The next lemma provides an estimate for the cost functions.
\begin{lemma}\label{lem:CompCostFunt}
Let Assumptions \ref{assump:SDECoefs} and \ref{assump:GammaPsi} be satisfied. Let $\widetilde{\Gamma}$ be a uniformly continuous real-valued function on $\overline{Q^{0}}\times\mathcal{U}$, such that $D_{x}\widetilde{\Gamma}$ is a bounded continuous function on $\overline{Q^{0}}\times\mathcal{U}$. For any $(t,x)\in Q^{0}$, any generalized reference probability space $\mu=(\Omega,\mathscr{F},\mathscr{F}_{s}^{t},\mathbb{P},\mathcal{W},\mathcal{L})$ and any control process $U\in\mathcal{A}_{\mu}$, let $X(s):=X(s;t,x)$ (respectively, $\widetilde{X}(s):=\widetilde{X}(s;t,x)$) be the unique strong c\`{a}dl\`{a}g solution to \eqref{eq:MainSDEs} (respectively, \eqref{eq:CompareSDEs}). Then,
\begin{align*}
&\mathbb{E}\left(\int_{t}^{T}\left|\Gamma\left(r,X(r),U(r)\right)-\widetilde{\Gamma}\left(r,\widetilde{X}(r),U(r)\right)\right|dr\right)\\
&\quad\leq (T-t)\bigg[\big\|\Gamma-\widetilde{\Gamma}\big\|_{L^{\infty}(\overline{Q^{0}}\times\mathcal{U})}+K_{4}\big\|D_{x}\widetilde{\Gamma}\big\|_{L^{\infty}(\overline{Q^{0}}\times\mathcal{U})}\bigg(\big\|b-\tilde{b}\big\|_{L^{\infty}(\overline{Q^{0}}\times\mathcal{U})}+\left\|\sigma-\widetilde{\sigma}\right\|_{L^{\infty}(\overline{Q^{0}}\times\mathcal{U})}\\
&\qquad\qquad\quad\,\,\,\,+\bigg(\int_{\mathbb{R}^{m_{2}}_{0}}\left\|\gamma(\cdot,\cdot,\cdot,z)-\widetilde{\gamma}(\cdot,\cdot,\cdot,z)\right\|_{L^{\infty}(\overline{Q^{0}}\times\mathcal{U})}^{2}\nu(dz)\bigg)^{1/2}\bigg)\bigg],
\end{align*}
where $K_{4}=K_{4}(C,T,M)>0$ is a constant depending only on $C$, $T$ and $M$.
\end{lemma}
\begin{proof} Note that for any $s\in[t,T]$,
\begin{align*}
\left|\Gamma\!\left(s,X(s),U(s)\right)-\widetilde{\Gamma}\!\left(s,\widetilde{X}(s),U(s)\right)\right|\leq\big\|\Gamma-\widetilde{\Gamma}\big\|_{L^{\infty}(\overline{Q^{0}}\times\mathcal{U})}+\big\|D_{x}\widetilde{\Gamma}\big\|_{L^{\infty}( \overline{Q^{0}}\times\mathcal{U})}\!\sup_{s\in[t,T]}\left|X(s)-\widetilde{X}(s)\right|.
\end{align*}
So the result follows immediately from Lemma \ref{lem:CompStrSols} and Cauchy-Schwarz inequality.
\end{proof}

\section{Representation Formulas and Dynamic Programming Principle}\label{sec:DP}

This section is devoted to the proof of the stochastic representation formulas and the Dynamic Programming Principle. Recall that Assumptions \ref{assump:SDECoefs}, \ref{assump:GammaPsi}, \ref{assump:Domain O}, and \ref{assump:ellpiticity along boundary} hold throughout the paper. For any $(t,x)\in Q^{0}$, any generalized reference probability space $\mu=(\Omega,\mathscr{F},\mathscr{F}_{s}^{t},\mathbb{P},\mathcal{W},\mathcal{L})$, and any $U\in\mathcal{A}_{\mu}$, let $X(s;t,x)$ be the unique strong c\`{a}dl\`{a}g solution to \eqref{eq:MainSDEs}.

For any generalized reference probability space $\mu=(\Omega,\mathscr{F},\mathscr{F}_{s}^{t},\mathbb{P},\mathcal{W},\mathcal{L})$ and any $U\in\mathcal{A}_{\mu}$, we choose an $\mathscr{F}_{s}^{t}$-stopping time $\theta_{U}$, with $\theta_{U}\in[t,T]$ $\mathbb{P}$-a.$\,$s. Let $\widetilde{\mathcal{A}}_{\mu}$ be the collection of all such pairs $(U,\theta_{U})$. We also define $\widetilde{\mathcal{A}}_{t}:=\cup_{\mu}\widetilde{\mathcal{A}}_{\mu}$, where the union is taken over all generalized reference probability spaces $\mu$ on $[t,T]$.

\subsection{Smooth Value Function}\label{Subsec:DPSmoothSols}

We first establish the Dynamic Programming Principle when there exists a classical solution to \eqref{eq:HJB} with terminal-boundary condition \eqref{eq:TermBoundCondHJB}.
\begin{theorem}\label{thm:DPSmooth}
Let Assumptions \ref{assump:SDECoefs} and \ref{assump:GammaPsi} be satisfied. Let $W\in C^{1,2}(\overline{Q^{0}})$ be a solution to \eqref{eq:HJB} with terminal-boundary condition \eqref{eq:TermBoundCondHJB}. Then, we have
\begin{equation}\label{eq:WVVmu}
W(t,x)=V(t,x)=V_{\mu}(t,x),\quad (t,x)\in \overline{Q},
\end{equation}
for any generalized reference probability space $\mu$, and
\begin{align}\label{eq:DP}
W(t,x)=\inf_{(U,\theta_{U})\in\widetilde{\mathcal{A}}_{\mu}}\mathbb{E}\left(\int_{t}^{\theta_{U}\wedge\tau}\Gamma\left(s,X(s;t,x),U(s)\right)ds+W\left(\theta_{U}\wedge\tau,X(\theta_{U}\wedge\tau;t,x)\right)\right).
\end{align}
\end{theorem}
\begin{proof}
Since $W\in C^{1,2}(\overline{Q^{0}})$, by It\^{o}'s formula and \eqref{eq:GenX}, for any generalized reference probability space $\mu$, and any $(U,\theta_{U})\in\widetilde{\mathcal{A}}_{\mu}$,
\begin{align}\label{eq:ItoW}
W(t,x)&=W\!\left(\theta_{U}\wedge\tau,X(\theta_{U}\wedge\tau)\right)-\int_{t}^{\theta_{U}\wedge\tau}\mathscr{A}^{U(s)}W\!\left(s,X(s)\right)ds\nonumber\\
&\quad -\int_{t}^{\theta_{U}\wedge\tau}D_{x}W(s,X(s))\cdot\sigma(s,X(s),U(s))\,d\mathcal{W}(s)\nonumber\\
&\quad -\int_{t}^{\theta_{U}\wedge\tau}\int_{\mathbb{R}_{0}^{m_{2}}}\left(W\left(s,X(s-)+\gamma(s,X(s-),U(s),z)\right)-W\left(s,X(s-)\right)\right)\widetilde{N}(ds,dz).
\end{align}
Taking the expectation of both sides of \eqref{eq:ItoW} gives
\begin{align*}
W(t,x)=\mathbb{E}\left(W\left(\theta_{U}\wedge\tau,X(\theta_{U}\wedge\tau)\right)-\int_{t}^{\theta_{U}\wedge\tau}\mathscr{A}^{U(s)}W\left(s,X(s)\right)ds\right).
\end{align*}
This, together with \eqref{eq:HJB}, yields
\begin{align*}
W(t,x)\leq\mathbb{E}\left(W\left(\theta_{U}\wedge\tau,X(\theta_{U}\wedge\tau)\right)+\int_{t}^{\theta_{U}\wedge\tau}\Gamma\left(s,X(s),U(s)\right)ds\right).
\end{align*}

To prove the reverse inequality, let us fix any $\kappa>0$. Since $W\in C^{1,2}(\overline{Q^{0}})$ and $\Gamma(\cdot,\cdot,u)$ is uniformly continuous on $\overline{Q^{0}}$, uniformly for $u\in\mathcal{U}$, there exists $\delta>0$ such that, for any $t_{1},t_{2}\in[t,T]$ and any $x_{1},x_{2}\in\overline{O}$, with $|t_{1}-t_{2}|<\delta$ and $|x_{1}-x_{2}|<\delta$, we have
\begin{align}\label{eq:UniContAWGamma}
\left|\mathscr{A}^{u}W(t_{1},x_{1})+\Gamma(t_{1},x_{1},u)-\mathscr{A}^{u}W(t_{2},x_{2})-\Gamma(t_{2},x_{2},u)\right|<\frac{\kappa}{2},\quad\text{for any }u\in\mathcal{U}.
\end{align}
Let $M_{1}\in\mathbb{N}$ be sufficiently large so that $(T-t)/M_{1}<\delta$. We partition $[t,T]$ into $M_{1}$ subintervals
$[t_0,t_1], t_0=t,$ and $(t_{j},t_{j+1}]$, $j=1,\ldots,M_{1}-1$, of length $(T-t)/M_{1}$. Moreover, let $\overline{O}=O_{1}\cup\cdots\cup O_{M_{2}}$, where $O_{1},\ldots,O_{M_{2}}$ are disjoint Borel sets of diameter less than $\delta/2$ and $M_{2}\in\mathbb{N}$. For each $k=1,\ldots,M_{2}$, we fix $x_{k}\in O_{k}$ arbitrarily. For each $j=0,\ldots,M_{1}-1$ and $k=1,\ldots,M_{2}$, since $W$ satisfies $(\ref{eq:HJB})$, there exists $u_{jk}\in \mathcal{U}$ such that
\begin{align*}
\mathscr{A}^{u_{jk}}W(t_{j},x_{k})+\Gamma(t_{j},x_{k},u_{jk})<\frac{\kappa}{2}.
\end{align*}
Together with \eqref{eq:UniContAWGamma}, for any $(s,y)\in [t_{j},t_{j+1}]\times (B_{\delta}(x_k)\cap\overline{O})$,
\begin{align}\label{eq:UpperBoundAujkWGamma}
\mathscr{A}^{u_{jk}}W(s,y)+\Gamma(s,y,u_{jk})<\kappa.
\end{align}

Define $\bar{u}=(\bar{u}_{0},\cdots,\bar{u}_{M_{1}-1}):\mathbb{R}^{d}\rightarrow\mathcal{U}^{M_{1}}$ by
\begin{align*}
\bar{u}_{j}(y):=\left\{\begin{array}{ll} u_{jk},\quad &\text{if }\,y\in O_{k}, \\ u_{0},\quad &\text{if }\,y\in\overline{O}^{c},\end{array}\right.,\quad j=0,\ldots,M_{1}-1,\,\,\,\,k=1,\ldots,M_{2},
\end{align*}
where $u_{0}\in\mathcal{U}$ is arbitrarily fixed. We define a Markov control policy and the corresponding solution to \eqref{eq:MainSDEs}, as follows. For $s\in [t,t_{1}]$, let $X(s;t,x)$ be the unique solution to \eqref{eq:MainSDEs} with control $U(s)\equiv\bar{u}_{0}(x)$. By induction, for $j=1,\ldots,M_{1}-1$, assume that $X(\,\cdot\,;t,x)$ and $U(\,\cdot\,)$ have been constructed on $[t,t_{j}]$. For $s\in(t_{j},t_{j+1}]$, let $U(s)=\bar{u}_{j}(X(t_{j};t,x))$, and let $X(s;t,x)$ be the associated solution to
\begin{align*}
X(s,t,x)&=X(t_{j};t,x)+\int_{t_{j}}^{s}b\left(r,X(r;t,x),U(r)\right)dr+\int_{t_{j}}^{s}\sigma\left(r,X(r;t,x),U(r)\right)d\mathcal{W}(r)\\
&\quad +\int_{t_{j}}^{s}\int_{\mathbb{R}_{0}^{m_{2}}}\gamma\left(r,X(r-;t,x),U(s),z\right)\widetilde{N}(dr,dz).
\end{align*}
Thus, $U(s)\equiv\bar{u}_{0}(x)$ for $s\in[t,t_{1}]$ and, for $j=1,\ldots,M_{1}-1$,
\begin{align}\label{eq:MarkovControl}
U(s)=u_{jk}\quad\text{if }\,s\in (t_{j},t_{j+1}]\,\,\,\,\text{and}\,\,\,\, X(t_{j};t,x)\in O_{k},\quad k=1,\ldots,M_{2}.
\end{align}
By \eqref{eq:ItoW}, we have
\begin{align}
W(t,x)&=\mathbb{E}\left(\int_{t}^{\theta_{U}\wedge\tau}\Gamma\left(s,X(s),U(s)\right)ds+W\left(\theta_{U}\wedge\tau,X(\theta_{U}\wedge\tau)\right)\right)\nonumber\\
\label{eq:ItoW2} &\quad -\mathbb{E}\left(\int_{t}^{\theta_{U}\wedge\tau}\left(\mathscr{A}^{U(s)}W(s,X(s))+\Gamma\left(s,X(s),U(s)\right)\right)ds\right).
\end{align}
For $j=0,\ldots,M_{1}-1$, let
\begin{align*}
\Omega_{j}:=\left\{\omega\in\Omega:\,\left|X(s\wedge\tau\wedge\theta_{U})(\omega)-X(t_{j}\wedge\tau\wedge\theta_{U})(\omega)\right|<\frac{\delta}{2},\,\,\,\,\text{for all }\,s\in[t_{j},t_{j+1}]\right\}.
\end{align*}
By \eqref{eq:UpperBoundAujkWGamma}, \eqref{eq:MarkovControl}, and the fact that $|X(t_{j})-x_{k}|<\delta/2$ if $ X(t_{j})\in O_{k}$, for any $s\in (t_{j},t_{j+1})$, $s\leq\tau\wedge\theta_{U}$,
\begin{align}\label{eq:UpperBoundAUWGammaMC}
\mathscr{A}^{U(s)}W\left(s,X(s)\right)+\Gamma\left(s,X(s),U(s)\right)<\kappa,\quad\text{in }\,\Omega_{j}.
\end{align}
By Lemma \ref{lem:2MomentEst}, for some constant $K_{5}>0$ depending only on $C$, $T$ and $M$,
\begin{align*}
\mathbb{P}(\Omega_{j}^{c})=\mathbb{P}\left(\sup_{s\in [t_{j},t_{j+1}]}\left|X(s\wedge\tau\wedge\theta_{U})-X(t_{j}\wedge\tau\wedge\theta_{U})\right|\geq\frac{\delta}{2}\right)\leq\frac{4K_{5}}{\delta^{2}}\left(t_{j+1}-t_{j}\right).
\end{align*}
Hence, \eqref{eq:UpperBoundAUWGammaMC} leads to
\begin{align*}
&\mathbb{E}\left(\int_{t}^{\theta_{U}\wedge\tau}\left(\mathscr{A}^{U(s)}W(s,X(s))+\Gamma\left(s,X(s),U(s)\right)\right)ds\right)\\
&\quad =\sum_{j=0}^{M_{1}-1}\mathbb{E}\left(\int_{t_{j}\wedge\tau}^{\theta_{U}\wedge t_{j+1}\wedge\tau}\left(\mathscr{A}^{U(s)}W(s,X(s))+\Gamma\left(s,X(s),U(s)\right)\right)ds\right)\\
&\quad\leq\kappa(T-t)+\sum_{j=0}^{M_{1}-1}\left\|\mathscr{A}^{u}W+\Gamma\right\|_{L^{\infty}([t,T]\times\overline{O}\times\mathcal{U})}(t_{j+1}-t_{j})\,\mathbb{P}(\Omega_{j}^{c})
\end{align*}
\begin{align*}
&\quad\leq\kappa(T-t)+\left\|\mathscr{A}^{u}W+\Gamma\right\|_{L^{\infty}([t,T]\times\overline{O}\times\mathcal{U})}\frac{4 K_{5}(T-t)^2}{\delta^{2}M_{1}}.
\end{align*}
Since $\kappa$ and $M_{1}$ are arbitrary, this, together with \eqref{eq:ItoW2}, gives us
\begin{align*}
W(t,x)\geq\inf_{(U,\theta_{U})\in\widetilde{\mathcal{A}}_{\mu}}\mathbb{E}\left(\int_{t}^{\theta_{U}\wedge\tau}\Gamma\left(s,X(s),U(s)\right)ds+W\left(\theta_{U}\wedge\tau,X(\theta_{U}\wedge\tau)\right)\right),
\end{align*}
which completes the proof of \eqref{eq:DP}. Finally, by choosing the constant stopping times $\theta_{U}\equiv T$ in \eqref{eq:DP} and noting that \eqref{eq:DP} is independent of the choice of a generalized reference probability space, we obtain \eqref{eq:WVVmu}.
\end{proof}

We point out that the construction of almost optimal controls in the proof of Theorem \ref{thm:DPSmooth} applied to the
case of Section \ref{Subsec:DPVisSolFinCtrl}, together with the uniform convergence of the value functions for the approximating control problems, provides a recipe how to construct $\varepsilon$-optimal controls for the original stochastic optimal control problem associated with our equation.

\subsection{Finite Control Sets}\label{Subsec:DPVisSolFinCtrl}

In this subsection, we assume that $\mathcal{U}$ is a finite set, which will be relaxed later by an approximation argument. For any $\delta\in(0,\eta/2)$, recall that $O_{\delta}=\{y\in\mathbb{R}^{d}:\,\text{dist}(y,O)<\delta\}$. We define $Q_{\delta}:=[0,T)\times O_{\delta}$.
In the sequel, $o_{\xi}(1)$ denotes any function of $\xi\in\mathbb{R}$ which converges to 0 as $\xi\rightarrow 0$.

Using Assumptions \ref{assump:SDECoefs}, \ref{assump:GammaPsi}, \ref{assump:Domain O}, \ref{assump:ellpiticity along boundary}, and Theorem \ref{thm:proxdomain}, we can construct sequences of functions $b_{n}:\overline{Q^{0}}\times\mathcal{U}\rightarrow\mathbb{R}^{d}$, $\sigma_{n}:\overline{Q^{0}}\times\mathcal{U}\rightarrow\mathbb{R}^{d\times m_{1}}$, $\gamma_{n}:\overline{Q^{0}}\times\mathcal{U}\times\mathbb{R}^{m_{2}}\rightarrow\mathbb{R}^{d}$, and $\Gamma_{n}:\overline{Q^{0}}\times\mathcal{U}\rightarrow\mathbb{R}$, $n\in\mathbb{N}$, satisfying the following assumptions.
\begin{assumption}\label{assump:CondsApproxFunts}
\begin{itemize}
\item [(i)] There exists a universal constant $\widetilde{C}>0$, and for any $n\in\mathbb{N}$, there exists a constant $\widetilde{C}_{n}>0$, depending only on $n$, such that for any $(t_{1},x_{1}),(t_{2},x_{2})\in \overline{Q^{0}}$, $z\in\mathbb{R}^{m_{2}}$, and $u\in\mathcal{U}$,
    \begin{align*}
    &\left|b_{n}(t_{1},x_{1},u)\!-\!b_{n}(t_{2},x_{2},u)\right|+\left\|\sigma_{n}(t_{1},x_{1},u)\!-\!\sigma_{n}(t_{2},x_{2},u)\right\|+\left|\Gamma_{n}(t_{1},x_{1},u)\!-\!\Gamma_{n}(t_{2},x_{2},u)\right|\\
    &\quad\leq\widetilde{C}_{n}\left(\left|t_{1}-t_{2}\right|+\left|x_{1}-x_{2}\right|\right),\\
    &\qquad\qquad\left|\gamma_{n}(t_{1},x_{1},u,z)-\gamma_{n}(t_{2},x_{2},u,z)\right|\leq\widetilde{C}_{n}\,\rho(z)\left(\left|t_{1}-t_{2}\right|+\left|x_{1}-x_{2}\right|\right),\\
    &\left\|b_{n}\right\|_{L^{\infty}(\overline{Q^{0}}\times\mathcal{U})}+\left\|\sigma_{n}\right\|_{L^{\infty}(\overline{Q^{0}}\times\mathcal{U})}+\left\|\Gamma_{n}\right\|_{L^{\infty}(\overline{Q^{0}}\times\mathcal{U})}\leq\widetilde{C},\quad\left\|\gamma_{n}(\cdot,\cdot,\cdot,z)\right\|_{L^{\infty}(\overline{Q^{0}}\times\mathcal{U})}\leq\widetilde{C}\rho(z).
    \end{align*}
\item [(ii)] As $n\rightarrow\infty$,
    \begin{align*}
    \max\left(\left\|b-b_{n}\right\|_{L^{\infty}(\overline{Q^{0}}\times\mathcal{U})},\,\left\|\sigma-\sigma_{n}\right\|_{L^{\infty}(\overline{Q^{0}}\times\mathcal{U})},\,\|\Gamma-\Gamma_{n}\|_{L^{\infty}( \overline{Q^{0}}\times\mathcal{U})}\right)=o_{1/n}(1),\quad\\
    \int_{\mathbb{R}^{m_{2}}_{0}}\left\|\gamma(\cdot,\cdot,\cdot,z)-\gamma_{n}(\cdot,\cdot,\cdot,z)\right\|_{L^\infty(\overline{Q^{0}}\times\mathcal{U})}^{2}\nu(dz)=o_{1/n}(1),\qquad\qquad\quad\\
    \left\|D_{x}\Gamma_{n}\right\|_{L^{\infty}(\overline{Q^{0}}\times\mathcal{U})}\cdot\max\left(\left\|b-b_{n}\right\|_{L^{\infty}( \overline{Q^{0}}\times\mathcal{U})},\,\left\|\sigma-\sigma_{n}\right\|_{L^{\infty}( \overline{Q^{0}}\times\mathcal{U})}\right)=o_{1/n}(1),\quad\\
    \left\|D_{x}\Gamma_{n}\right\|_{L^{\infty}(\overline{Q^{0}}\times\mathcal{U})}\left(\int_{\mathbb{R}^{m_{2}}_{0}}\left\|\gamma(\cdot,\cdot,\cdot,z)-\gamma_{n}(\cdot,\cdot,\cdot,z)\right\|_{L^\infty(\overline{Q^{0}}\times\mathcal{U})}^{2}\nu(dz)\right)^{1/2}=o_{1/n}(1).
    \end{align*}
\item [(iii)] For any sufficiently small $\delta>0$ and any $x\in\partial O_{\delta}$, there exists a unit vector $n_{x,\delta}\in N(O_{\delta},x)$, such that $\overline{B}_{\eta/2}(x+\eta n_{x,\delta}/2)\cap\overline{O}_{\delta}=\{x\}$. Moreover, for any $t\in[0,T]$, $u\in\mathcal{U}$, and any $n\in\mathbb{N}$ large enough,
    \begin{align*}
    n_{x,\delta}\,\sigma_{n}(t,x,u)\sigma_{n}^{T}(t,x,u)\,n_{x,\delta}^{T}\geq\frac{\lambda}{2}.
    \end{align*}
\end{itemize}
\end{assumption}

Next, for each arbitrarily fixed $t\in[0,T)$, we consider an {\it extended generalized reference probability space} $\mu_{1}=(\Omega,\mathscr{F},\mathscr{F}_{s}^{t},\mathbb{P},\mathcal{W},\widetilde{\mathcal{W}},\mathcal{L})$, where the probability space is large enough to accommodate another standard $d$-dimensional $\mathscr{F}_{s}^{t}$-Brownian motion $\widetilde{\mathcal{W}}$, which is independent of $\mathcal{W}$ and $\mathcal{L}$. Let $\mathcal{A}_{\mu_{1}}$ be the collection of all $\mathscr{F}_{s}^{t}$-predictable $\mathcal{U}$-valued processes on $\mu_{1}$, and let $\mathcal{A}_{t}^{e}:=\cup_{\mu_1}\mathcal{A}_{\mu_{1}}$, where the union is taken over all extended generalized reference probability spaces $\mu_{1}$.
\begin{remark}\label{rem:EquiProbSpaces}
If $\mu_{1}=(\Omega,\mathscr{F},\mathscr{F}_{s}^{t},\mathbb{P},\mathcal{W},\widetilde{\mathcal{W}},\mathcal{L})$ is an extended generalized reference probability space, then $\mu=(\Omega,\mathscr{F},\mathscr{F}_{s}^{t},\mathbb{P},\mathcal{W},\mathcal{L})$ is a generalized reference probability space, and clearly we have $\mathcal{A}_{\mu}=\mathcal{A}_{\mu_1}$. On the other hand, given a generalized reference probability space $\mu=(\Omega,\mathscr{F},\mathscr{F}_{s}^{t},\mathbb{P},\mathcal{W},\mathcal{L})$, consider a standard $d$-dimensional
$\widetilde{\mathscr{F}_{s}^{t}}$-Brownian motion $\widetilde{\mathcal{W}}$ defined on a filtered probability space
$(\widetilde{\Omega},\widetilde{\mathscr{F}},\widetilde{\mathscr{F}_{s}^{t}},\widetilde{\mathbb{P}})$. For $(\omega,\widetilde{\omega})\in\Omega_{1}:=\Omega\times\widetilde{\Omega}$, let
\begin{align*}
\mathcal{W}_{1}(s)(\omega,\widetilde{\omega})=\mathcal{W}(s)(\omega),\quad\mathcal{W}_{2}(s)(\omega,\widetilde{\omega})=\widetilde{\mathcal{W}}(s)(\widetilde{\omega}),\quad\mathcal{L}_{1}(s)(\omega,\widetilde{\omega})=\mathcal{L}(s)(\omega).
\end{align*}
Then
\begin{align*}
\mu_{1}:=\left(\Omega_{1},\overline{\mathscr{F}\otimes\widetilde{\mathscr{F}}},\mathscr{F}_{1,s}^{t},\mathbb{P}\otimes\widetilde{\mathbb{P}},\mathcal{W}_{1},\mathcal{W}_{2},\mathcal{L}_{1}\right),
\end{align*}
where $\overline{\mathscr{F}\otimes\widetilde{\mathscr{F}}}$ is the augmentation of the $\sigma$-field $\mathscr{F}\otimes\widetilde{\mathscr{F}}$ by the $\mathbb{P}\otimes\widetilde{\mathbb{P}}$ null sets, and $\mathscr{F}_{1,s}^{t}:=\overline{\cap_{r>s}\mathscr{F}_{s}^{t}\otimes\widetilde{\mathscr{F}_{s}^{t}}}$, is an extended generalized reference probability space, and any element $U\in\mathcal{A}_{\mu}$ can be regarded as an element in $\mathcal{A}_{\mu_{1}}$. Thus we have $\mathcal{A}_{t}^{e}=\mathcal{A}_{t}$.
\end{remark}

Let $\{\epsilon_{n}\}_{n\in\mathbb{N}}$ be a positive sequence of real numbers such that $\epsilon_{n}\rightarrow 0$, as $n\rightarrow\infty$. For any extended generalized reference probability space $\mu_{1}=(\Omega,\mathscr{F},\mathscr{F}_{s}^{t},\mathbb{P},\mathcal{W},\widetilde{\mathcal{W}},\mathcal{L})$, any $U\in\mathcal{A}_{\mu_{1}}$, any $x\in\mathbb{R}^{d}$, and any $n\in\mathbb{N}$, consider an $\mathbb{R}^{d}$-valued stochastic process $X_{n}(s;t,x)$ which is the solution to the following controlled SDE:
\begin{align*}
X_{n}(s;t,x)&=x+\int_{t}^{s}b_{n}\left(r,X_{n}(r;t,x),U(r)\right)dr+\int_{t}^{s}\sigma_{n}\left(r,X_{n}(r;t,x),U(r)\right)d\mathcal{W}(r)\\
&\quad +\int_{t}^{s}\sqrt{\epsilon_{n}}\,d\widetilde{\mathcal{W}}(r)+\int_{t}^{s}\!\int_{\mathbb{R}^{m_{2}}_{0}}\gamma_{n}\left(r,X_{n}(r-;t,x),U(r),z\right)\widetilde{N}(dr,dz),\quad s\in[t,T].
\end{align*}
A similar argument as in Theorem \ref{thm:StrSolsSDEs} ensures that the above SDE has a unique strong solution $X_{n}(s;t,x)$ with $\mathbb{P}-$a.$\,$s. c\`{a}dl\`{a}g sample paths. For any $\delta\in(0,\eta/2)$ and $(t,x)\in Q^{0}$, let
\begin{align*}
\tau_{\delta}=\tau_{\delta}(t,x)&:=\inf\left\{s\in[t,T]:\,X(s;t,x)\not\in O_{\delta/2}\right\},\\
\tau_{\delta,n}=\tau_{\delta,n}(t,x)&:=\inf\left\{s\in[t,T]:\,X_{n}(s;t,x)\not\in O_{\delta/2}\right\},
\end{align*}
with the convention $\inf\emptyset=T$.
\begin{lemma}\label{lem:ProbSndelta}
Let Assumption \ref{assump:SDECoefs} be satisfied. Let $\{b_{n}\}_{n\in\mathbb{N}}$, $\{\sigma_{n}\}_{n\in\mathbb{N}}$, and $\{\gamma_{n}\}_{n\in\mathbb{N}}$ be the sequences satisfying Assumption \ref{assump:CondsApproxFunts}-(i) {\rm $\&$} (ii). For any $x\in O$, let
\begin{align*}
\mathcal{S}_{\delta,n}=\mathcal{S}_{\delta,n}(t,x):=\left\{\omega\in\Omega:\sup_{\ell\in[t,T]}\left|X\!\left(\ell\!\wedge\!(\tau_{\delta}\!\vee\!\tau_{\delta,n});t,x\right)\!(\omega)-X_{n}\!\left(\ell\!\wedge\!(\tau_{\delta}\!\vee\!\tau_{\delta,n});t,x\right)\!(\omega)\right|>\frac{\delta}{2}\right\}.
\end{align*}
Then, we have
\begin{align}\label{eq:ProbSdeltan}
\mathbb{P}(\mathcal{S}_{\delta,n})\rightarrow 0,\quad\text{as }\,n\rightarrow\infty.
\end{align}
Moreover, for any $\omega\in\mathcal{S}_{\delta,n}^{c}$,
\begin{align}\label{eq:tautaudeltanSdeltan}
\tau(w)\wedge\tau_{\delta,n}(w)=\tau(w).
\end{align}
\end{lemma}
\begin{proof}
Convergence \eqref{eq:ProbSdeltan} is a direct consequence of Lemma \ref{lem:CompStrSols}, Chebyshev's inequality and Assumption \ref{assump:CondsApproxFunts}-(ii), while \eqref{eq:tautaudeltanSdeltan} follows from the definition of $\mathcal{S}_{\delta,n}$.
\end{proof}

We first assume that $\Psi$ is more regular, i.e., $\Psi\in C^{1+\alpha/2,\,2+\alpha}(\overline{Q^{0}})$ for some $\alpha>0$. We will remove the regularity assumption on $\Psi$ at the end of this section. We obtain the following existence, uniqueness and regularity theorem using a result proved in \cite{Mou1:2016}.
\begin{theorem}\label{thm:uer}
Let $\mathcal{U}$ be a finite set, and let Assumption \ref{assump:Domain O} be valid. Let $\{b_{n}\}_{n\in\mathbb{N}}$, $\{\sigma_{n}\}_{n\in\mathbb{N}}$, $\{\gamma_{n}\}_{n\in\mathbb{N}}$, and $\{\Gamma_{n}\}_{n\in\mathbb{N}}$ be the sequences satisfying Assumption \ref{assump:CondsApproxFunts}-(i), and let $\Psi\in C^{1+\alpha/2,\,2+\alpha}(\overline{Q^{0}})$ for some $\alpha>0$ small enough. Then, there exists a unique viscosity solution
\begin{align*}
W_{\delta,n}\in C^{1+\alpha/2,2+\alpha}_{\rm loc}(Q_{\delta})\,\cap\,{\rm Lip}_{b}\left(\overline{Q^{0}}\right)
\end{align*}
to
\begin{align}\label{eq:HJBDeltan}
\inf_{u\in\mathcal{U}}\left(\mathscr{A}_{n}^{u}W_{\delta,n}(t,x)+\Gamma_{n}(t,x,u)\right)=0\quad\text{in }\,Q_{\delta},
\end{align}
with terminal-boundary condition
\begin{align*}
W_{\delta,n}(t,x)=\Psi(t,x),\quad (t,x)\in\partial_{{\rm np}}Q_{\delta},
\end{align*}
where for every $u\in\mathcal{U}$,
\begin{align*}
\mathscr{A}_{n}^{u}W_{\delta,n}(t,x)&:=\frac{\partial W_{\delta,n}}{\partial t}(t,x)+b_{n}(t,x,u)\cdot D_{x}W_{\delta,n}(t,x)+\frac{1}{2}{\rm tr}\!\left(\left(a_{n}(t,x,u)+\epsilon_{n}I\right)D_{x}^{2}W_{\delta,n}(t,x)\right)\\ &\quad\,+\int_{\mathbb{R}^{m_{2}}_{0}}\left(W_{\delta,n}\left(t,x+\gamma_{n}(t,x,u,z)\right)-W_{\delta,n}(t,x)-D_{x}W_{\delta,n}(t,x)\cdot\gamma_{n}(t,x,u,z)\right)\nu(dz),
\end{align*}
and where $a_{n}(t,x,u):=\sigma_{n}(t,x,u)\sigma_{n}^{T}(t,x,u)$.
\end{theorem}

\begin{proof}
Since $\delta\in (0,\eta/2)$, by Theorem \ref{thm:proxdomain}, $O_{\delta}$ satisfies the uniform exterior ball condition with a uniform radius $\eta/2$. It is easy to verify that all the coefficients $b_{n}$, $\sigma_{n}$, $\gamma_{n}$ and the boundary data $\Psi$ satisfy the same regularity and boundedness conditions as required in~\cite[Theorem 5.3]{Mou1:2016}. Since $a_{n}+\epsilon_{n}I=\sigma_{n}\sigma_{n}^{T}+\epsilon_{n}I\geq\epsilon_{n}I$, the operator $\mathcal{A}_{n}^u$ is uniformly parabolic in $Q_{\delta}$. The result follows immediately from \cite[Theorem 5.3]{Mou1:2016}.
\end{proof}
\begin{theorem}\label{thm:wdeltan}
Under the assumptions of Theorem \ref{thm:uer}, for any $x\in\mathbb{R}^{d}$,
\begin{align*}
W_{\delta,n}(t,x)&=\inf_{U\in\mathcal{A}_{\mu_{1}}}\mathbb{E}\left(\int_{t}^{\tau(t,x)\wedge\tau_{\delta,n}(t,x)}\Gamma_{n}\left(s,X_{n}(s;t,x),U(s)\right)ds\right.\\
&\qquad\qquad\quad\,\,\,+W_{\delta,n}\left(\tau(t,x)\wedge\tau_{\delta,n}(t,x),X_{n}(\tau(t,x)\wedge\tau_{\delta,n}(t,x);t,x)\right)\bigg).
\end{align*}
\end{theorem}
\begin{proof}
By Theorem \ref{thm:uer}, $W_{\delta,n}\in C^{1+\alpha/2,2+\alpha}_{\rm loc}(Q_{\delta})\cap{\rm Lip}_{b}\left(\overline{Q^{0}}\right)$ is a classical solution to \eqref{eq:HJBDeltan}. Then, there exists a sequence of functions $\{W_{\delta,n,m}\}_{m\in\mathbb N}$ such that $W_{\delta,n,m}\equiv W_{\delta,n}$ in $[0,T]\times \overline{\Omega_{3\delta/4}}$, $W_{\delta,n,m}\rightarrow W_{\delta,n}$ uniformly in $\overline{Q^{0}}$ as $m\rightarrow\infty$, and $W_{\delta,n,m}\in C^{1+\alpha/2,2+\alpha}([0,\ell]\times \mathbb R^d)$ for any fixed $\ell\in (t,T)$. We notice that $W_{\delta,n,m}$ satisfies a different equation, which is
\begin{eqnarray*}
\inf_{u\in\mathcal{U}}\left(\mathscr{A}_{n}^{u}W_{\delta,n,m}(t,x)+\Gamma_{n,m}(t,x,u)\right)=0\quad\text{in }\,[0,T)\times O_{\delta/2},
\end{eqnarray*}
where
\begin{align*}
&\Gamma_{n,m}(t,x,u)=\Gamma_{n}(t,x,u)+\int_{\mathbb{R}^{m_{2}}_{0}}\left(W_{\delta,n}\left(t,x+\gamma_{n}(t,x,u,z)\right)-W_{\delta,n,m}\left(t,x+\gamma_{n}(t,x,u,z)\right)\right)\nu(dz)
\\
&=\Gamma_{n}(t,x,u)+\int_{\{z\in \mathbb R_0^{m_2}:\,|\gamma_n(t,x,u,z)|\geq \delta/4\}}\left|W_{\delta,n}\left(t,x+\gamma_{n}(t,x,u,z)\right)-W_{\delta,n,m}\left(t,x+\gamma_{n}(t,x,u,z)\right)\right|\nu(dz).
\end{align*}
Since $W_{\delta,n,m}\in C^{1+\alpha/2,2+\alpha}([0,\ell]\times\mathbb R^d)$, applying Theorem $\ref{thm:DPSmooth}$ with $\theta_{U}=\tau_{\delta,n}\wedge\ell$, we have
\begin{align*}
W_{\delta,n,m}(t,x)&=\inf_{U\in\mathcal{A}_{\mu_{1}}}\mathbb{E}\left(\int_{t}^{\tau\wedge\tau_{\delta,n}
\wedge\ell}\Gamma_{n,m}\left(s,X_{n}(s;t,x),U(s)\right)ds\right.\nonumber\\
&\qquad\qquad\quad\,\,\,+W_{\delta,n,m}\left(\tau\wedge\tau_{\delta,n}\wedge\ell,X_{n}(\tau\wedge\tau_{\delta,n}
\wedge\ell;t,x)\right)\bigg).
\end{align*}
We claim that $\Gamma_{n,m}\to \Gamma_n$ uniformly in $\overline{Q^{0}}\times U$. We first notice that
\begin{equation*}
\frac{\delta^2}{16{\widetilde{C}}^2}\int_{\{z\in \mathbb R_0^{m_2}:\,\widetilde{C}\rho(z)\geq \delta/4\}}\nu(dz)\leq \int_{\{z\in \mathbb R_0^{m_2}:\,\widetilde{C}\rho(z)\geq \delta/4\}}\rho^2(z)\nu(dz)\leq\int_{\mathbb R_0^{m_2}}\rho^2(z)\nu(dz),
\end{equation*}
where $\widetilde{C}$ is from Assumption \ref{assump:CondsApproxFunts}-(i).
Using the above inequality and Assumption \ref{assump:CondsApproxFunts}-(i), we have for any $(t,x,u)\in [0,T)\times O_{\delta/2}\times U$
\begin{align*}
&\int_{\{z\in \mathbb R_0^{m_2}:\,|\gamma_n(t,x,u,z)|\geq \delta/4\}}\left|W_{\delta,n}\left(t,x+\gamma_{n}(t,x,u,z)\right)-W_{\delta,n,m}\left(t,x+\gamma_{n}(t,x,u,z)\right)\right|\nu(dz)\\
\leq&\int_{\{z\in \mathbb R_0^{m_2}:\,\widetilde{C}\rho(z)\geq \delta/4\}}\left|W_{\delta,n}\left(t,x+\gamma_{n}(t,x,u,z)\right)-W_{\delta,n,m}\left(t,x+\gamma_{n}(t,x,u,z)\right)\right|\nu(dz)\\
\leq& o_{1/m}(1)\int_{\{z\in \mathbb R_0^{m_2}:\,\widetilde{C}\rho(z)\geq \delta/4\}}\nu(dz)\leq o_{1/m}(1)\frac{16{\widetilde{C}}^2}{\delta^2}\int_{\mathbb R_0^{m_2}}\rho^2(z)\nu(dz).
\end{align*}
Letting $m\rightarrow\infty$ in both sides of the dynamic programming equality we thus get
\begin{align*}
W_{\delta,n}(t,x)&=\inf_{U\in\mathcal{A}_{\mu_{1}}}\mathbb{E}\left(\int_{t}^{\tau\wedge\tau_{\delta,n}
\wedge\ell}\Gamma_{n}\left(s,X_{n}(s;t,x),U(s)\right)ds\right.\nonumber\\
&\qquad\qquad\quad\,\,\,+W_{\delta,n}\left(\tau\wedge\tau_{\delta,n}\wedge\ell,X_{n}(\tau\wedge\tau_{\delta,n}
\wedge\ell;t,x)\right)\bigg).
\end{align*}
It remains to use Lemma \ref{lem:2MomentEst} and let $\ell\to T$ to conclude the proof.
\end{proof}


It is well known that, under Assumptions \ref{assump:SDECoefs} and \ref{assump:GammaPsi}, comparison principle holds for
equation
\begin{align}\label{eq:HJBDelta}
\inf_{u\in\mathcal{U}}\left(\mathscr{A}^{u}W_{\delta}(t,x)+\Gamma(t,x,u)\right)=0\quad\text{in }Q_{\delta},
\end{align}
with the terminal-boundary condition
\begin{align*}
W_{\delta}(t,x)=\Psi(t,x),\quad (t,x)\in\partial_{\text{np}}Q_{\delta}.
\end{align*}
Moreover, under Assumptions \ref{assump:Domain O} and \ref{assump:ellpiticity along boundary}, the above parabolic Dirichlet problem admits a unique viscosity solution $W_{\delta}\in C_{b}(\overline{Q^{0}})$. The same results hold when $Q_{\delta}$ is replaced by $Q$. We refer the reader to, e.g., \cite[Theorem 3]{BarlesImbert:2008} and~\cite[Theorem 3.1]{JakobsenKarlsen:2006}, for proofs of comparison principle, and to, e.g., \cite[Theorem 3.2]{Mou:2017} and~\cite[Theorem 5.1]{Mou:2018}, for proofs of the existence results.

\medskip
For any $(t,x)\in \overline{Q^{0}}$, let
\begin{align*}
\widetilde{W}_{\delta}(t,x)&:=\lim_{k\rightarrow\infty}\sup\left\{ W_{\delta,n}(s,y):\,\,n\geq k,\,s\in[0,T]\cap\left[t-\frac{1}{k},t+\frac{1}{k}\right],\,y\in\overline{B}_{1/k}(x)\right\},\\
\widetilde{W}^{\delta}(t,x)&:= \lim_{k\rightarrow\infty}\inf\left\{W_{\delta,n}(s,y):\,\,n\geq k,\,s\in[0,T]\cap\left[t-\frac{1}{k},t+\frac{1}{k}\right],\,y\in\overline{B}_{1/k}(x)\right\}.
\end{align*}
\begin{lemma}\label{lem:VisSubSuperSolsWDelta}
Let the assumptions of Theorem \ref{thm:uer} be satisfied. Let $\{b_{n}\}_{n\in\mathbb{N}}$, $\{\sigma_{n}\}_{n\in\mathbb{N}}$, $\{\gamma_{n}\}_{n\in\mathbb{N}}$ and $\{\Gamma_{n}\}_{n\in\mathbb{N}}$ also satisfy Assumption \ref{assump:CondsApproxFunts}-(ii). Then, the function $\widetilde{W}_{\delta}$ (respectively, $\widetilde{W}^{\delta}$) is a viscosity subsolution (respectively, supersolution) to $\eqref{eq:HJBDelta}$.
\end{lemma}
\begin{proof}
We will only present the proof for $\widetilde{W}_{\delta}$ as the proof for $\widetilde{W}^{\delta}$ is similar. Suppose that $\widetilde{W}_{\delta}-\varphi$ has a maximum (equal to $0$) over $\overline{Q^{0}}$ at some $(t_{0},x_{0})\in Q_{\delta}$, for a test function $\varphi\in C_{b}^{1,2}(\overline{Q^{0}})$. By appropriate approximation and modification of $\varphi$, we can assume, without loss of generality, that the maximum is strict and
\begin{align*}
\sup_{(t,x)\in\partial_{\text{np}}Q_{\delta}}\left(\widetilde{W}_{\delta}(t,x)-\varphi(t,x)\right)=\sup_{(t,x)\in\partial_{\text{np}}Q_\delta}\left(\Psi(t,x)-\varphi(t,x)\right)\leq c_{0}<0,
\end{align*}
for some constant $c_{0}<0$. Hence, there exists a modulus of continuity $\varpi_{1}$ such that, for any $\varepsilon>0$,
\begin{align*}
\sup_{(t,x)\in B_{\varepsilon}^{c}(t_{0},x_{0})\cap \overline{Q^{0}}}\left(\widetilde{W}_{\delta}(t,x)-\varphi(t,x)\right)\leq -\varpi_{1}(\varepsilon)<0.
\end{align*}

Next, for any $(t,x)\in \overline{Q_{\delta}}$, by the definition of $\widetilde{W}_{\delta}$, there exists $k_{0}(t,x):=k_{0}(t,x;\varepsilon)\in\mathbb{N}$, such that for any $n\geq k_{0}(t,x)$,
\begin{align*}
\sup_{\substack{s\in[0,T],\,|s-t|\leq 1/k_{0}(t,x) \\ |y-x|\leq 1/k_{0}(t,x)}}W_{\delta,n}(s,y)-\widetilde{W}_{\delta}(t,x)<\frac{\varpi_{1}(\varepsilon)}{4}.
\end{align*}
Since $\varphi\in C^{1,2}(\overline{Q^{0}})$, $\varphi$ is uniformly continuous in $[0,T]\times\overline{O}_{1}$ with a modulus of continuity $\varpi_{2}$. Hence, there exists $\eta_{0}:=\eta_{0}(\varepsilon)>0$ such that, for any $(t,x)\in\overline{Q}_{\delta}\setminus B_{\varepsilon}(t_{0},x_{0})$ and any $n\geq k_{0}(t,x)$,
\begin{align*}
&\sup_{\substack{s\in[0,T],\,|s-t|\leq(1/k_{0}(t,x))\wedge\eta_{0} \\ |y-x|\leq(1/k_{0}(t,x))\wedge\eta_{0}}}\left(W_{\delta,n}(s,y)-\varphi(s,y)\right)\\
&\quad\leq\sup_{\substack{s\in[0,T],\,|s-t|\leq(1/k_{0}(t,x))\wedge\eta_{0} \\ |y-x|\leq(1/k_{0}(t,x))\wedge\eta_{0}}}\left(W_{\delta,n}(s,y)-\widetilde{W}_{\delta}(t,x)+\widetilde{W}_{\delta}(t,x)-\varphi(t,x)+\varphi(t,x)-\varphi(s,y)\right)\\
&\quad\leq\frac{\varpi_{1}(\varepsilon)}{4}-\varpi_{1}(\varepsilon)+\varpi_{2}(\eta_{0})\leq\frac{\varpi_{1}(\varepsilon)}{4}-\varpi_{1}(\varepsilon)+\frac{\varpi_{1}(\varepsilon)}{4}=-\frac{\varpi_{1}(\varepsilon)}{2}.
\end{align*}
Since $\overline{Q}_{\delta}\setminus B_{\varepsilon}( t_{0},x_{0})$ is a compact set, and since $\{B_{(1/k(t,x))\wedge\eta_{0}}(t,x)\}_{(t,x)\in\overline{Q}_{\delta}\setminus B_{\varepsilon}(t_{0},x_{0})}$ is a cover of $\overline{Q}_{\delta}\setminus B_{\varepsilon}(t_{0},x_{0})$, there exist $N=N(\varepsilon)\in\mathbb{N}$ and $(s_{i},y_{i})=(s_{i}(\varepsilon),y_{i}(\varepsilon))\in\overline{Q}_{\delta}\setminus B_{\varepsilon}(t_{0},x_{0})$, $i=1,\ldots,N$, such that $\{B_{(1/k(s_{i},y_{i}))\wedge\eta_{0}}(s_{i},y_{i})\}_{i=1}^N$ is a finite cover of $\overline{Q}_{\delta}\setminus B_{\varepsilon}(t_{0},x_{0})$. Hence, for any $n\geq \max_{1\leq i\leq N}k(s_{i},y_{i})$ and any $(t,x)\in\overline{Q}_{\delta}\setminus B_{\varepsilon}(t_{0},x_{0})$,
\begin{equation*}
W_{\delta,n}(t,x)-\varphi(t,x)\leq -\frac{\varpi_{1}(\varepsilon)}{2}.
\end{equation*}

Finally, by the definition of $\widetilde{W}_{\delta}$, for any positive sequence $\{\varepsilon_{\ell}\}_{\ell\in\mathbb{N}}$ with $\varepsilon_{\ell}\downarrow 0$, as $\ell\rightarrow\infty$, there exists $(t_{\ell},x_{\ell})\in\overline{Q}_{\delta}\cap B_{\varepsilon_{\ell}}(t_{0},x_{0})$ and $n_{\ell}\geq\max_{1\leq i\leq N(\varepsilon_{\ell})}k(s_{i}(\varepsilon_{\ell}),y_{i}(\varepsilon_{\ell}))$, where $n_{\ell}\uparrow\infty$ as $\ell\rightarrow\infty$, such that
\begin{align*}
W_{\delta,n_{\ell}}(t_{\ell},x_{\ell})-\varphi(t_{\ell},x_{\ell})=\max_{Q^{0}}\left(W_{\delta,n_{\ell}}(t,x)-\varphi(t,x)\right)> -\frac{\varpi_{1}(\varepsilon_{\ell})}{2}.
\end{align*}
Therefore, we have
\begin{align*}
\inf_{u\in\mathcal{U}}\left(\mathscr{A}_{n_{\ell}}^{u}\varphi(t_{\ell},x_{\ell})+\Gamma_{n_{\ell}}(t_{\ell},x_{\ell},u)\right)\geq 0.
\end{align*}
Letting $\ell\rightarrow\infty$, we have
\begin{align*}
\inf_{u\in\mathcal{U}}\left(\mathscr{A}^{u}\varphi(t_{0},x_{0})+\Gamma(t_{0},x_{0},u)\right)\geq 0,
\end{align*}
which completes the proof of the lemma.
\end{proof}

\begin{lemma}\label{lem:ConvWdeltanWdelta}
Let $\mathcal{U}$ be a finite set, and let Assumptions \ref{assump:SDECoefs}, \ref{assump:GammaPsi}, \ref{assump:Domain O}, and \ref{assump:ellpiticity along boundary} be valid. Let $\{b_{n}\}_{n\in\mathbb{N}}$, $\{\sigma_{n}\}_{n\in\mathbb{N}}$, $\{\gamma_{n}\}_{n\in\mathbb{N}}$, and $\{\Gamma_{n}\}_{n\in\mathbb{N}}$ be the sequences satisfying Assumption \ref{assump:CondsApproxFunts}, and let $\Psi\in C^{1+\alpha/2,\,2+\alpha}(\overline{Q^{0}})$ for some $\alpha>0$. Then, both $\eqref{eq:HJBDeltan}$ and $\eqref{eq:HJBDelta}$ have the unique viscosity solutions $W_{\delta,n}$ and $W_{\delta}$, respectively, and $\|W_{\delta,n}-W_{\delta}\|_{L^{\infty}(\overline{Q^{0}})}\rightarrow 0$, as $n\rightarrow\infty$.
\end{lemma}
\begin{proof}
By Lemma \ref{lem:barfun1} there exist functions $\psi_{\delta}$ and $\psi^{\delta}$ which are respectively a viscosity subsolution and a viscosity supersolution to \eqref{eq:HJBDeltan} and $\psi_{\delta}=\psi^{\delta}=\Psi$ in $\partial_{{\rm np}}Q_\delta$. We have $\psi_{\delta}\leq W_{\delta,n}\leq\psi^{\delta}$, for any $n\in\mathbb{N}$, by the comparison principle. Next, since $\psi_{\delta},\psi^{\delta}\in C(\overline{Q^{0}})$, it follows that $\psi_{\delta}\leq\widetilde{W}_{\delta}\leq\psi^{\delta}$ and
$\psi_{\delta}\leq\widetilde{W}^{\delta}\leq\psi^{\delta}$. By Lemma \ref{lem:VisSubSuperSolsWDelta} and the comparison principle, we have $\widetilde{W}_{\delta}\leq\widetilde{W}^{\delta}$. By the definitions of $\widetilde{W}_{\delta}$ and $\widetilde{W}^{\delta}$, we also have $\widetilde{W}_{\delta}\geq\widetilde{W}^{\delta}$. Hence, we obtain $\widetilde{W}^{\delta}=\widetilde{W}_{\delta}=W_{\delta}\in C(\overline{Q^{0}})$.

It is now standard to notice that $\|W_{\delta,n}-W_{\delta}\|_{L^{\infty}(\overline{Q^{0}})}\rightarrow 0$, as $n\rightarrow\infty$. Otherwise, there would exist an $\varepsilon_{0}>0$, $\{n_{k}\}_{k\in\mathbb{N}}\subset\mathbb{N}$ with $n_{k}\uparrow\infty$, as $k\rightarrow\infty$, and $\{(t_{k},x_{k})\}_{k\in\mathbb{N}}\subset Q_{\delta}$, such that
\begin{align*}
\left|W_{\delta,n_{k}}(t_{k},x_{k})-W_{\delta}(t_{k},x_{k})\right|>\varepsilon_{0}.
\end{align*}
Without loss of generality, we can assume that there exists $(t_{0},x_{0})\in\overline{Q}_{\delta}$, such that $(t_{k},x_{k})\rightarrow (t_{0},x_{0})$, as $k\rightarrow\infty$. Letting $k\rightarrow\infty$, we have either $\widetilde{W}_{\delta}(t_{0},x_{0})-W_{\delta}(t_{0},x_{0})\geq\varepsilon_{0}$ or $\widetilde{W}^{\delta}(t_{0},x_{0})-W_{\delta}(t_{0},x_{0})\leq -\varepsilon_{0}$, which contradicts with the fact that $\widetilde{W}_{\delta}=\widetilde{W}^{\delta}=W_{\delta}$ in $\overline{Q^{0}}$.
\end{proof}

The next result provides a representation formula for $W_{\delta}$ with a finite control set.
\begin{theorem}\label{thm:DPWdeltaFin}
Let $\mathcal{U}$ be a finite set, and let Assumptions \ref{assump:SDECoefs}, \ref{assump:GammaPsi}, \ref{assump:Domain O}, and \ref{assump:ellpiticity along boundary} be valid. Let $\Psi\in C^{1+\alpha/2,\,2+\alpha}(\overline{Q^{0}})$ for some small $\alpha>0$. For each $t\in[0,T]$, let $\mu_{1}=(\Omega,\mathscr{F},\mathscr{F}_{s}^{t},\mathbb{P},\mathcal{W},\widetilde{\mathcal{W}},\mathcal{L})$ be an extended generalized reference probability space and set $\mu=(\Omega,\mathscr{F},\mathscr{F}_{s}^{t},\mathbb{P},\mathcal{W},\mathcal{L})$. Then, for any $x\in \overline O$,
\begin{align}\label{eq:DPWdeltaFin}
W_{\delta}(t,x)=\inf_{U\in\mathcal{A}_{\mu}}\mathbb{E}\left(\int_{t}^{\tau(t,x)}\Gamma\left(s,X(s;t,x),U(s)\right)ds+W_{\delta}(\tau(t,x),X(\tau;t,x))\right).
\end{align}
\end{theorem}
\begin{proof}
Let $\{b_{n}\}_{n\in\mathbb{N}}$, $\{\sigma_{n}\}_{n\in\mathbb{N}}$, $\{\gamma_{n}\}_{n\in\mathbb{N}}$, and $\{\Gamma_{n}\}_{n\in\mathbb{N}}$ be sequences of functions satisfying Assumption \ref{assump:CondsApproxFunts}. By Theorem \ref{thm:DPSmooth}, we have
\begin{align} W_{\delta,n}(t,x)&=\inf_{U\in\mathcal{A}_{\mu}}\mathbb{E}\left(\int_{t}^{\tau(t,x)\wedge\tau_{\delta,n}(t,x)}\Gamma_{n}\left(s,X_{n}(s;t,x),U(s)\right)ds\right.\nonumber\\
\label{eq:DPWdeltan} &\qquad\qquad\quad\,\,\,+W_{\delta,n}\left(\tau(t,x)\wedge\tau_{\delta,n}(t,x),X_{n}(\tau(t,x)\wedge\tau_{\delta,n}(t,x);t,x)\right)\bigg).
\end{align}
Notice that, by Assumption \ref{assump:CondsApproxFunts}-(i) and the construction of the HJB equation \eqref{eq:HJBDeltan}, there exists a constant $K>0$, independent of $n$, such that
\begin{align*}
\sup_{n\in\mathbb{N}}\left\|\Gamma_{n}\right\|_{L^{\infty}(\overline{Q_{\delta}}\times\mathcal{U})}+\sup_{n\in\mathbb{N}}\left\|W_{\delta,n}\right\|_{L^{\infty}(\overline{Q^{0}})}\leq K.
\end{align*}
By Lemma \ref{lem:CompCostFunt} and Lemma \ref{lem:ProbSndelta}, for any $ U\in\mathcal{A}_{\mu}$,
\begin{align*}
&\mathbb{E}\left(\left|\int_{t}^{\tau\wedge\tau_{\delta,n}}\Gamma_{n}\left(s,X_{n}(s),U(s)\right)ds-\int_{t}^{\tau}\Gamma\left(s,X(s),U(s)\right)ds\right|\right)\\
&\quad\leq\mathbb{E}\left(\int_{t}^{\tau\wedge\tau_{\delta,n}}\left|\Gamma_{n}\left(s,X_{n}(s),U(s)\right)-\Gamma\left(s,X(s),U(s)\right)\right|ds+\int_{\tau\wedge\tau_{\delta,n}}^{\tau}\left|\Gamma\left(s,X(s),U(s)\right)\right|ds\right)\\
&\quad\leq (T-t)\left[\left\|\Gamma_{n}-\Gamma\right\|_{L^{\infty}(\overline{Q^{0}}\times\mathcal{U})}+K_{4}\left\|D_{x}\Gamma_{n}\right\|_{L^{\infty}(\overline{Q^{0}}\times\mathcal{U})}\left(\left\|b-b_{n}\right\|_{L^{\infty}(\overline{Q^{0}}\times\mathcal{U})}+\left\|\sigma-\sigma_{n}\right\|_{L^{\infty}(\overline{Q^{0}}\times\mathcal{U})}\right)\right]\\
&\qquad +K_{4}(T-t)\left\|D_{x}\Gamma_{n}\right\|_{L^{\infty}(\overline{Q^{0}}\times\mathcal{U})}\left(\int_{\mathbb{R}^{m_{2}}_{0}}\left\|\gamma(\cdot,\cdot,\cdot,z)-\gamma_n(\cdot,\cdot,\cdot,z)\right\|_{L^{\infty}(\overline{Q^{0}}\times\mathcal{U})}^{2}\nu(dz)\right)^{1/2}\\
&\qquad +\mathbb{E}\left({\bf 1}_{\mathcal{S}_{\delta,n}}\int_{\tau\wedge\tau_{\delta,n}}^{\tau}\left|\Gamma\left(s,X(s),U(s)\right)\right|ds\right)\\
&\quad\leq o_{1/n}(1).
\end{align*}
Similarly, by Lemma \ref{lem:ProbSndelta} and Lemma \ref{lem:ConvWdeltanWdelta}, we also have
\begin{align*}
&\mathbb{E}\left(\left|W_{\delta,n}\left(\tau\wedge\tau_{\delta,n},X_{n}(\tau\wedge\tau_{\delta,n}\right)-W_{\delta}\left(\tau,X(\tau)\right)\right|\right)\\
&\quad\leq\mathbb{E}\left(\left|W_{\delta,n}\left(\tau\wedge\tau_{\delta,n},X_{n}(\tau\wedge\tau_{\delta,n})\right)-W_{\delta}\left(\tau\wedge\tau_{\delta,n},X_{n}(\tau\wedge\tau_{\delta,n})\right)\right|\right)\\
&\qquad\,+\mathbb{E}\left(\left|W_{\delta}\left(\tau\wedge\tau_{\delta,n},X_{n}(\tau\wedge\tau_{\delta,n})\right)-W_{\delta}\left(\tau,X(\tau)\right)\right|\right)\\
&\quad\leq\left\|W_{\delta,n}-W_{\delta}\right\|_{L^{\infty}(\overline{Q^{0}})}+\mathbb{E}\left({\bf 1}_{\mathcal{S}_{\delta,n}^{c}}\left|W_{\delta}\left(\tau,X_{n}(\tau)\right)-W_{\delta}\left(\tau,X(\tau)\right)\right|\right)+2\|W_{\delta}\|_{L^{\infty}(\overline{Q^{0}})}\mathbb{P}(\mathcal{S}_{\delta,n})\\
&\quad\leq o_{1/n}(1)+\mathbb{E}\left(\left|W_{\delta}\left(\tau,X_{n}(\tau)\right)-W_{\delta}\left(\tau,X(\tau)\right)\right|\right)\\
&\quad\leq o_{1/n}(1)+\mathbb{E}\left(\varpi_{\delta}\left(|X_n(\tau)-X(\tau)|\right)\right),
\end{align*}
where $\varpi_{\delta}$ is a (concave) modulus of continuity of $W_\delta$ in $\overline{Q^0}$. By Jensen's inequality and Lemma \ref{lem:CompStrSols}, we thus obtain
\begin{equation*}
\mathbb{E}\left(\left|W_{\delta,n}\left(\tau\wedge\tau_{\delta,n},X_{n}(\tau\wedge\tau_{\delta,n}\right)-W_{\delta}\left(\tau,X(\tau)\right)\right|\right)\leq o_{1/n}(1).
\end{equation*}
Therefore, \eqref{eq:DPWdeltaFin} follows immediately by letting $n\rightarrow\infty$ in \eqref{eq:DPWdeltan}.
\end{proof}
\begin{remark}\label{rem:ImprovDPFinite}
By almost the same arguments we can prove, under the assumptions of Theorem \ref{thm:DPWdeltaFin}, the following version of the Dynamic Programming Principle
\begin{equation*}
W_\delta(t,x)=\inf_{(U,\theta_{U})\in\widetilde{\mathcal{A}}_{\mu}}\mathbb{E}\left(\int_{t}^{\theta_{U}\wedge\tau}\Gamma\left(s,X(s;t,x),U(s)\right)ds+W_{\delta}\left(\theta_{U}\wedge\tau,X(\theta_{U}\wedge\tau;t,x)\right)\right),
\end{equation*}
for any generalized reference probability space $\mu=(\Omega,\mathscr{F},\mathscr{F}_{s}^{t},\mathbb{P},\mathcal{W},\mathcal{L})$ which comes from an extended generalized reference probability space $\mu_{1}=(\Omega,\mathscr{F},\mathscr{F}_{s}^{t},\mathbb{P},\mathcal{W},\widetilde{\mathcal{W}},\mathcal{L})$, and any $(t,x)\in\overline{Q}$.
\end{remark}

\subsection{General Control Sets}\label{Subsec:DPVisSolGenCtrl}

In this subsection, we consider the general control space, i.e., $\mathcal{U}$ is a Polish space. Let $\{v_{i}\}_{i\in\mathbb{N}}$ be a countable dense subset of $\mathcal{U}$. For each $n\in\mathbb{N}$, let $\mathcal{U}_{n}:=\{v_{1},\ldots,v_{n}\}$, and for each extended generalized reference probability space $\mu_{1}=(\Omega,\mathscr{F},\mathscr{F}_{s}^{t},\mathbb{P},\mathcal{W},\widetilde{\mathcal{W}},\mathcal{L})$, we set $\mu=(\Omega,\mathscr{F},\mathscr{F}_{s}^{t},\mathbb{P},\mathcal{W},\mathcal{L})$, and let $\mathcal{A}_{\mu}^{n}$ be the collection of all $\mathscr{F}_{s}^{t}$-predictable $\mathcal{U}_{n}$-valued processes on $[t,T]$. For any $U_{n}\in\mathcal{A}_{\mu}^{n}$ and any $x\in\mathbb{R}^{d}$, we denote by $\overline{X}_{n}(s;t,x)$ the unique strong c\`{a}dl\`{a}g solution to
\begin{align}
\overline{X}_{n}(s;t,x)&=x+\int_{t}^{s}b\left(r,\overline{X}_{n}(r;t,x),U_{n}(r)\right)dr+\int_{t}^{s}\sigma\left(r,\overline{X}_{n}(r;t,x),U_{n}(s)\right)d\mathcal{W}(s)\nonumber\\
\label{eq:MainSDEsUn} &\quad+\int_{t}^{s}\int_{\mathbb{R}_{0}^{m_{2}}}\gamma\left(r,\overline{X}_{n}(r-;t,x),U_{n}(s),z\right)\widetilde{N}(dr,dz),\quad s\in[t,T].
\end{align}
For any $\delta\in(0,\eta/2)$ and $(t,x)\in Q^0$,
\begin{align*}
\overline{\tau}_{\delta,n}=\overline{\tau}_{\delta,n}(t,x):=\inf\left\{s\in[t,T]:\,\overline{X}_{n}(s;t,x)\not\in O_{\delta/2}\right\},
\end{align*}
with the convention $\inf\emptyset=T$.
\begin{lemma}\label{lem:ApproControl}
Let Assumptions \ref{assump:SDECoefs} and \ref{assump:GammaPsi} be satisfied. Let
$\mu_{1}=(\Omega,\mathscr{F},\mathscr{F}_{s}^{t},\mathbb{P},\mathcal{W},\widetilde{\mathcal{W}},\mathcal{L})$ be an extended generalized reference probability space, and let $\mu=(\Omega,\mathscr{F},\mathscr{F}_{s}^{t},\mathbb{P},\mathcal{W},\mathcal{L})$. For any $U\in\mathcal{A}_{\mu}$, there exists a sequence of control processes $\{U_{n_{k}}\}_{k\in\mathbb{N}}$, where $U_{n_{k}}\in\mathcal{A}_{\mu}^{n_{k}}$ for each $k\in\mathbb{N}$, such that for any $x\in\mathbb{R}^{d}$,
\begin{align}\label{eq:ApproxGammaFinControl}
\mathbb{E}\left(\int_{t}^{T}\left|\Gamma\left(s,X(s;t,x),U_{n_{k}}(s)\right)-\Gamma\left(s,X(s;t,x),U(s)\right)\right|^{2}ds\right)=o_{1/k}(1),\quad k\rightarrow\infty.
\end{align}
Moreover,
\begin{align}\label{eq:ApproxSDESolsFinControl}
\mathbb{E}\left(\sup_{\ell\in[t,T]}\left|X(\ell;t,x)-\overline{X}_{n_{k}}(\ell;t,x)\right|^{2}\right)=o_{1/k}(1),\quad k\rightarrow\infty.
\end{align}
\end{lemma}
\begin{proof}
By Assumption \ref{assump:SDECoefs}-(ii)(iii) and Assumption \ref{assump:GammaPsi}-(ii), for each $k\in\mathbb{N}$, there exists $\delta_{k}>0$, such that for any $u_{1},u_{2}\in\mathcal{U}$ with $d_{\mathcal{U}}(u_{1},u_{2})\leq\delta_{k}$,
\begin{align}\label{eq:EstDiffbsigmaU}
\left|b(s,y,u_{1})-b(s,y,u_{2})\right|+\left\|\sigma(s,y,u_{1})-\sigma(s,y,u_{2})\right\|\leq\frac{1}{k},\\
\label{eq:EstDiffgammaU} \left|\gamma(s,y,u_{1},z)-\gamma(s,y,u_{2},z)\right|\leq\frac{\rho(z)}{k},\qquad\quad\,\,\,\,\\
\label{eq:EstDiffGammaU} \left|\Gamma(s,y,u_{1})-\Gamma(s,y,u_{2})\right|\leq\frac{1}{k},\qquad\qquad\quad\,
\end{align}
for any $(s,y)\in[t,T]\times\mathbb{R}^{d}$ and any $z\in\mathbb{R}^{m_{2}}_{0}$. Next, since $\{v_{i}\}_{i\in\mathbb{N}}$ is a countable dense subset of $\mathcal{U}$, clearly, $\mathcal{U}\subset\cup_{i\in\mathbb{N}}B_{\delta_{k}}(v_{i})$. It follows that, for any $U\in\mathcal{A}_{\mu}$, we have $[t,T]\times\Omega\subset\cup_{i\in\mathbb{N}}U^{-1}(B_{\delta_{k}}(v_{i}))$. Thus, there exists an increasing sequence of integers $\{n_{k}\}_{k\in\mathbb{N}}$, with $n_{k}\uparrow\infty$ as $k\rightarrow\infty$, such that, defining for each $k\in\mathbb{N}$, $A_{k}:=([t,T]\times\Omega)\setminus\bigcup_{i=1}^{n_{k}}U^{-1}(B_{\delta_{k}}(v_{i}))$, we have
\begin{align*}
\text{Leb}\otimes\mathbb{P}\left(A_{k}\right)\leq\frac{1}{k}.
\end{align*}
Fix any arbitrary element $u_{0}\in\mathcal{U}$. For each $k\in\mathbb{N}$, define the control policy $U_{n_{k}}$ via
\begin{align*}
U_{n_{k}}(s)(\omega)=\left\{\begin{array}{ll} v_{i} &\,\,\,\text{if }\,U(s)(\omega)\in B_{\delta_{k}}(v_{i})\setminus\left(\bigcup_{j=1}^{i-1}B_{\delta_{k}}(v_{j})\right),\quad i=1,\ldots,n_{k}, \\ u_{0} &\,\,\,\text{otherwise}. \end{array}\right.
\end{align*}
Clearly, for each $k\in\mathbb{N}$, $U_{n_{k}}\in\mathcal{A}_{\mu}^{n_{k}}$, and
\begin{align*}
d_{\mathcal{U}}\left(U_{n_{k}}(s)(\omega),U(s)(\omega)\right)\leq\delta_{k},\quad\text{for }\,(s,\omega)\in ([t,T]\times\Omega)\setminus A_{k}.
\end{align*}
Hence, by \eqref{eq:EstDiffGammaU},
\begin{align*}
&\mathbb{E}\left(\int_{t}^{T}\left|\Gamma\left(s,X(s;t,x),U_{n_{k}}(s)\right)-\Gamma\left(s,X(s;t,x),U(s)\right)\right|^{2}ds\right)\\
&\quad\leq\mathbb{E}\left(\int_{t}^{T}\left|\Gamma\left(s,X(s;t,x),U_{n_{k}}(s)\right)-\Gamma\left(s,X(s;t,x),U(s)\right)\right|^{2}{\bf 1}_{A_{k}}ds\right)\\
&\qquad\,+\mathbb{E}\left(\int_{t}^{T}\left|\Gamma\left(s,X(s;t,x),U_{n_{k}}(s)\right)-\Gamma\left(s,X(s;t,x),U(s)\right)\right|^{2}{\bf 1}_{A_{k}^{c}}ds\right)\\
&\quad\leq 4\left\|\Gamma\right\|_{L^{\infty}(\overline{Q^{0}}\times\mathcal{U})}^{2}\cdot\text{Leb}\otimes\mathbb{P}\left(A_{k}\right)+\frac{T}{k^{2}}\leq\frac{4}{k}\left\|\Gamma\right\|_{L^{\infty}(\overline{Q^{0}}\times\mathcal{U})}^{2}+\frac{T}{k^{2}}.
\end{align*}
Letting $k\rightarrow\infty$ in the last inequality above completes the proof of \eqref{eq:ApproxGammaFinControl}.

Moreover, for any $x\in\mathbb{R}^{d}$, $s\in[t,T]$, $k\in\mathbb{N}$, by Burkholder-Davis-Gundy and Cauchy-Schwarz inequalities, there exists a universal constant $\overline{\Lambda}_{1}>0$ such that, setting $X(s)=X(s;t,x), \overline{X}_{n_{k}}(s)=\overline{X}_{n_{k}}(s;t,x)$,
\begin{align*}
\mathbb{E}\!\left(\sup_{\ell\in[t,s]}\!\left|X(\ell)\!-\!\overline{X}_{n_{k}}(\ell)\right|^{2}\!\right)&\leq 3T\,\mathbb{E}\left(\int_{t}^{s}\left|b\left(r,X(r),U(r)\right)-b\left(r,\overline{X}_{n_{k}}(r),U_{n_{k}}(r)\right)\right|^{2}dr\right)\\
&\quad +3\overline{\Lambda}_{1}\,\mathbb{E}\left(\int_{t}^{s}\left\|\sigma\left(r,X(r),U(r)\right)-\sigma\left(r,\overline{X}_{n_{k}}(r),U_{n_{k}}(r)\right)\right\|^{2}dr\right)\\
&\quad +3\overline{\Lambda}_{1}\mathbb{E}\!\left(\int_{t}^{s}\!\!\int_{\mathbb{R}^{m_{2}}_{0}}\!\left|\gamma\!\left(r,\!X(r),\!U(r),\!z\right)\!-\!\gamma\!\left(r,\!\overline{X}_{n_{k}}\!(r),\!U_{n_{k}}\!(r),\!z\right)\right|^{2}\!\nu(dz)dr\!\right).
\end{align*}
By Assumption \ref{assump:SDECoefs}-(iii) and \eqref{eq:EstDiffbsigmaU},
\begin{align*}
&\mathbb{E}\left(\int_{t}^{s}\left|b\left(r,X(r),U(r)\right)-b\left(r,\overline{X}_{n_{k}}(r),U_{n_{k}}(r)\right)\right|^{2}dr\right)\\
&\quad\leq 3\,\mathbb{E}\left(\int_{t}^{s}\left|b\left(r,X(r),U(r)\right)-b\left(r,\overline{X}_{n_{k}}(r),U(r)\right)\right|^{2}dr\right)\\
&\qquad +3\,\mathbb{E}\left(\int_{t}^{s}\left|b\left(r,\overline{X}_{n_{k}}(r),U(r)\right)-b\left(r,\overline{X}_{n_{k}}(r),U_{n_{k}}(r)\right)\right|^{2}{\bf 1}_{A_{k}}\,dr\right)\\
&\qquad +3\,\mathbb{E}\left(\int_{t}^{s}\left|b\left(r,\overline{X}_{n_{k}}(r),U(r)\right)-b\left(r,\overline{X}_{n_{k}}(r),U_{n_{k}}(r)\right)\right|^{2}{\bf 1}_{A_{k}^{c}}\,dr\right)
\end{align*}
\begin{align*}
&\quad\leq 3C^{2}\int_{t}^{s}\mathbb{E}\left(\left|X(r)-\overline{X}_{n_{k}}(r)\right|^{2}\right)dr+12T\left\|b\right\|^{2}_{L^{\infty}( \overline{Q^{0}}\times\mathcal{U})}\text{Leb}\otimes\mathbb{P}(A_{k})+\frac{3T}{k^{2}}\\
&\quad\leq 3C^{2}\int_{t}^{s}\mathbb{E}\left(\sup_{\ell\in[t,r]}\left|X(\ell)-\overline{X}_{n_{k}}(\ell)\right|^{2}\right)dr+\frac{12C^{2}T}{k}+\frac{3T}{k^{2}}.
\end{align*}
Similarly, by Assumption \ref{assump:SDECoefs}-(iii) and \eqref{eq:EstDiffbsigmaU} again,
\begin{align*}
&\mathbb{E}\left(\int_{t}^{s}\left\|\sigma\left(r,X(r),U(r)\right)-\sigma\left(r,\overline{X}_{n_{k}}(r),U_{n_{k}}(r)\right)\right\|^{2}dr\right)\\
&\quad\leq 3C^{2}\int_{t}^{s}\mathbb{E}\left(\sup_{\ell\in[t,r]}\left|X(\ell)-\overline{X}_{n_{k}}(\ell)\right|^{2}\right)dr+\frac{12C^{2}T}{k}+\frac{3T}{k^{2}},
\end{align*}
and by Assumption \ref{assump:SDECoefs}-(iii) and \eqref{eq:EstDiffgammaU},
\begin{align*}
&\mathbb{E}\left(\int_{t}^{s}\int_{\mathbb{R}^{m_{2}}_{0}}\left|\gamma\left(r,X(r),U(r),z\right)-\gamma\left(r,\overline{X}_{n_{k}}(r),U_{n_{k}}(r),z\right)\right|^{2}\nu(dz)\,dr\right)\\
&\quad\leq 3C^{2}M\int_{t}^{s}\mathbb{E}\left(\sup_{\ell\in[t,r]}\left|X(\ell)-\overline{X}_{n_{k}}(\ell)\right|^{2}\right)dr+\frac{12C^{2}MT}{k}+\frac{3MT}{k^{2}}.
\end{align*}
Therefore, we obtain
\begin{align*}
&\mathbb{E}\left(\sup_{\ell\in[t,s]}\left|X(\ell)-\overline{X}_{n_{k}}(\ell)\right|^{2}\right)\\
&\quad\leq 9C^{2}\left(T+\overline{\Lambda}_{1}\right)(2+M)\left(\int_{t}^{s}\mathbb{E}\left(\sup_{\ell\in[t,r]}\left|X(\ell)-\overline{X}_{n_{k}}(\ell)\right|^{2}\right)dr+\frac{4C^{2}T}{k}+\frac{T}{k^{2}}\right),
\end{align*}
and \eqref{eq:ApproxSDESolsFinControl} follows immediately from Gronwall's inequality.
\end{proof}

Let $\Psi\in C^{1+\alpha/2,\,2+\alpha}(\overline{Q^{0}})$ for some small $\alpha>0$. From the arguments after the proof of Theorem \ref{thm:wdeltan}, under Assumptions \ref{assump:SDECoefs}, \ref{assump:GammaPsi}, \ref{assump:Domain O}, and \ref{assump:ellpiticity along boundary}, we know that the equation
\begin{align}\label{eq:HJBTildeDeltan}
\inf_{u\in\mathcal{U}_{n}}\left(\mathscr{A}^{u}\overline{W}_{\delta,n}(t,x)+\Gamma(t,x,u)\right)=0\quad\text{in }\,Q_{\delta},
\end{align}
with terminal-boundary condition
\begin{align}\label{eq:TermBoundCondHJBDeltan1}
\overline{W}_{\delta,n}(t,x)=\Psi(t,x),\quad (t,x)\in\partial_{\text{np}}Q_{\delta},
\end{align}
admits a unique viscosity solution $\overline{W}_{\delta,n}\in C_{b}(\overline{Q^{0}})$. Moreover, the functions $\{\overline{W}_{\delta,n}\}_{n\in\mathbb{N}}$ are uniformly bounded by the construction in~\cite[Theorem 3.2]{Mou:2017}. Similarly,
\begin{align}\label{eq:HJBDelta1}
\inf_{u\in\mathcal{U}}\left(\mathscr{A}^{u}\widehat{W}_{\delta}(t,x)+\Gamma(t,x,u)\right)=0\quad\text{in }\,Q_{\delta},
\end{align}
with terminal-boundary condition \eqref{eq:TermBoundCondHJBDeltan1}, admits a unique viscosity solution $\widehat{W}_{\delta}\in C_{b}(\overline{Q^{0}})$. Notice that the existence results in \cite{Mou1:2016} allow for a general (infinite) control space $\mathcal{U}$.

\smallskip
For any $(t,x)\in \overline{Q^{0}}$, let
\begin{align*}
\overline{W}_{\delta}(t,x)&:=\lim_{k\rightarrow\infty}\sup\left\{\overline{W}_{\delta,n}(s,y):\,n\geq k,\,s\in[0,T]\cap\left[t-\frac{1}{k},t+\frac{1}{k}\right],\,y\in\overline{B}_{1/k}(x)\right\},\\
\overline{W}^{\delta}(t,x)&:=\lim_{k\rightarrow\infty}\inf\left\{\overline{W}_{\delta,n}(s,y):\,n\geq k,\,s\in[0,T]\cap\left[t-\frac{1}{k},t+\frac{1}{k}\right],\,y\in\overline{B}_{1/k}(x)\right\}.
\end{align*}
\begin{lemma}\label{lem:VisSubSuperSolsOverlineWDelta}
Let Assumptions \ref{assump:SDECoefs}, \ref{assump:GammaPsi}, \ref{assump:Domain O}, and \ref{assump:ellpiticity along boundary} be valid. Let $\Psi\in C^{1+\alpha/2,2+\alpha}(\overline{Q^{0}})$ for some $\alpha>0$. Then $\overline{W}_{\delta}$ (respectively, $\overline{W}^{\delta}$) is a viscosity subsolution (respectively, supersolution) to \eqref{eq:HJBDelta1}.
\end{lemma}
\begin{proof}
The proof is similar to the proof of Lemma \ref{lem:VisSubSuperSolsWDelta}, and we only sketch it for $\overline{W}_{\delta}$. If $\overline{W}_{\delta}-\varphi$ has a strict global maximum over $\overline {Q^{0}}$ (equal to $0$) at some $(t_{0},x_{0})\in Q_{\delta}$ for some test function $\varphi\in C_{b}^{1,2}(\overline{Q^{0}})$, repeating the arguments from the proof of Lemma \ref{lem:VisSubSuperSolsWDelta}, we obtain for any positive sequence $\{\varepsilon_{\ell}\}_{\ell\in\mathbb{N}}$, with $\varepsilon_{\ell}\downarrow 0$ as $\ell\rightarrow\infty$, points $(t_{\ell},x_{\ell})\in Q_\delta\cap B_{\varepsilon_{\ell}}(t_{0},x_{0})$ and $n_{\ell}\in\mathbb{N}$, satisfying $n_{\ell}\uparrow\infty$ as $\ell\rightarrow\infty$, such that
\begin{align*}
\inf_{u\in\mathcal{U}_{n_{\ell}}}\left(\mathscr{A}^{u}\varphi(t_{\ell},x_{\ell})+\Gamma(t_{\ell},x_{\ell},u)\right)\geq 0.
\end{align*}
Letting $\ell\rightarrow\infty$ completes the proof of the lemma.
\end{proof}

The following result is an analog to Lemma \ref{lem:ConvWdeltanWdelta} above. The proof is very similar to that of Lemma \ref{lem:ConvWdeltanWdelta}, and is thus omitted.
\begin{lemma}\label{lem:ConvOverlineWdeltanWdelta}
let Assumptions \ref{assump:SDECoefs}, \ref{assump:GammaPsi}, \ref{assump:Domain O}, and \ref{assump:ellpiticity along boundary} be valid, and let $\Psi\in C^{1+\alpha/2,\,2+\alpha}(\overline{Q^{0}})$ for some $\alpha>0$. Then, both \eqref{eq:HJBTildeDeltan} and \eqref{eq:HJBDelta1} have the unique viscosity solutions $\overline{W}_{\delta,n}$ and $\widehat W_{\delta}$, respectively, and $\|\overline{W}_{\delta,n}-\widehat{W}_{\delta}\|_{L^{\infty}(\overline{Q^{0}})}\rightarrow 0$, as $n\rightarrow\infty$.
\end{lemma}

The following theorem provides a representation formula for $\widehat{W}_{\delta}$.
\begin{theorem}\label{thm:DPWdelta}
Let Assumptions \ref{assump:SDECoefs}, \ref{assump:GammaPsi}, \ref{assump:Domain O}, and \ref{assump:ellpiticity along boundary} be valid, and let $\Psi\in C^{1+\alpha/2,\,2+\alpha}(\overline{Q^{0}})$ for some $\alpha>0$. Let $t\in[0,T)$, let $\mu_{1}=(\Omega,\mathscr{F},\mathscr{F}_{s}^{t},\mathbb{P},\mathcal{W},\widetilde{\mathcal{W}},\mathcal{L})$
be an extended generalized reference probability space, and set $\mu=(\Omega,\mathscr{F},\mathscr{F}_{s}^{t},\mathbb{P},\mathcal{W},\mathcal{L})$. Then, for any $x\in\overline{O}$,
\begin{align}\label{eq:DPWdelta}
\widehat{W}_{\delta}(t,x)=\inf_{U\in\mathcal{A}_{\mu}}\mathbb{E}\left(\int_{t}^{\tau(t,x)}\Gamma\left(s,X(s;t,x),U(s)\right)ds+\widehat{W}_{\delta}\left(\tau(t,x),X(\tau(t,x);t,x)\right)\right).
\end{align}
\end{theorem}
\begin{proof}
For any $\eta>0$, there exists $U^{\eta}\in\mathcal{A}_{\mu}$, such that
\begin{align}
&\mathbb{E}\left(\int_{t}^{\tau^{\eta}}\Gamma\left(s,X^{\eta}(s),U^{\eta}(s)\right)ds+\widehat{W}_{\delta}\left(\tau^{\eta},X^{\eta}(\tau^{\eta})\right)\right)\nonumber\\
\label{eq:Approxeta} &\quad\leq\inf_{U\in\mathcal{A}_{\mu}}\mathbb{E}\left(\int_{t}^{\tau}\Gamma\left(s,X(s),U(s)\right)ds+\widehat{W}_{\delta}\left(\tau,X(\tau)\right)\right)+\eta,
\end{align}
where $X^{\eta}(s)=X^\eta(s;t,x)$ is the unique strong c\`{a}dl\`{a}g solution to \eqref{eq:MainSDEs} with control process $U^{\eta}$, and where (with $\inf\emptyset=T$)
\begin{align*}
\tau^{\eta}=\tau^{\eta}(t,x):=\inf\left\{s\in[t,T]:\,X^{\eta}(s;t,x)\not\in O\right\}.
\end{align*}
By Lemma \ref{lem:ApproControl}, there exists a sequence of increasing integers $\{n_{k}\}_{k\in\mathbb{N}}$, and a corresponding sequence of control processes $\{U^{\eta}_{n_{k}}\}_{k\in\mathbb{N}}$, where for each $k\in\mathbb{N}$, $U^{\eta}_{n_{k}}\in\mathcal{A}_{\mu}^{n_{k}}$, such that,
\begin{align*}
\mathbb{E}\left(\sup_{\ell\in[t,T]}\left|X\left(\ell\wedge\tau^{\eta}\wedge\overline{\tau}^{\eta}_{\delta,n_{k}}\right)-\overline{X}^{\eta}_{n_{k}}\left(\ell\wedge\tau^{\eta}\wedge\overline{\tau}^{\eta}_{\delta,n_{k}}\right)\right|^{2}\right)=o_{1/k}(1),\quad k\rightarrow\infty,
\end{align*}
where $\overline{X}^{\eta}_{n_{k}}(s;t,x)$ is the unique strong c\`{a}dl\`{a}g solution to \eqref{eq:MainSDEsUn} with control process $U^{\eta}_{n_{k}}$, and where (with $\inf\emptyset=T$)
\begin{align*}
\overline{\tau}^{\eta}_{\delta,n_{k}}=\overline{\tau}^{\eta}_{\delta,n_{k}}(t,x):=\inf\left\{s\in[t,T]:\,\overline{X}_{n_{k}}^{\eta}(s;t,x)\not\in O_{\delta/2}\right\}.
\end{align*}
By Remark \ref{rem:ImprovDPFinite}, we have
\begin{align}\label{eq:OverlineWdeltankleq}
\overline{W}_{\delta,n_{k}}(t,x)\leq\mathbb{E}\left(\int_{t}^{\tau^{\eta}\wedge\overline{\tau}_{\delta,n_{k}}^{\eta}}\Gamma\!\left(s,\overline{X}_{n_{k}}^{\eta}(s),U_{n_{k}}^{\eta}(s)\right)ds+\overline{W}_{\delta,n_{k}}\!\left(\tau^{\eta}\!\wedge\!\overline{\tau}_{\delta,n_{k}}^{\eta},\overline{X}_{n_{k}}^{\eta}\!\left(\tau^{\eta}\!\wedge\!\overline{\tau}_{\delta,n_{k}}^{\eta}\right)\right)\!\right).
\end{align}

We need to take the limits in both sides of \eqref{eq:OverlineWdeltankleq}, first as $k\rightarrow\infty$ (for fixed $\eta>0$) and then $\eta\rightarrow 0$. We have
\begin{align*}
&\mathbb{E}\left(\left|\int_{t}^{\tau^{\eta}\wedge\overline{\tau}_{\delta,n_{k}}^{\eta}}\Gamma\left(s,\overline{X}_{n_{k}}^{\eta}(s),U_{n_{k}}^{\eta}(s)\right)ds-\int_{t}^{\tau^{\eta}}\Gamma\left(s,X^{\eta}(s),U^{\eta}(s)\right)ds\right|\right)\\
&\quad\leq\mathbb{E}\left(\int_{t}^{\tau^{\eta}\wedge\overline{\tau}_{\delta,n_{k}}^{\eta}}\left|\Gamma\left(s,\overline{X}_{n_{k}}^{\eta}(s),U_{n_{k}}^{\eta}(s)\right)-\Gamma\left(s,X^{\eta}(s),U_{n_{k}}^{\eta}(s)\right)\right|ds\right)\\
&\qquad+\mathbb{E}\!\left(\int_{t}^{\tau^{\eta}\wedge\overline{\tau}_{\delta,n_{k}}^{\eta}}\!\left|\Gamma\!\left(s,\!X^{\eta}(s),\!U_{n_{k}}^{\eta}(s)\right)\!-\!\Gamma\!\left(s,\!X^{\eta}(s),\!U^{\eta}(s)\right)\right|\!ds\!\right)\!+\mathbb{E}\!\left(\int_{\tau^{\eta}\wedge\overline{\tau}_{\delta,n_{k}}^{\eta}}^{\tau^{\eta}}\!\!\!\!\left|\Gamma\!\left(s,\!X^{\eta}(s),\!U^{\eta}(s)\right)\right|\!ds\!\right).
\end{align*}
By Assumption \ref{assump:GammaPsi}-(ii) and Lemma \ref{lem:ApproControl}, for fixed $\eta>0$, as $k\rightarrow\infty$,
\begin{align*}
\mathbb{E}\left(\int_{t}^{\tau^{\eta}\wedge\overline{\tau}_{\delta,n_{k}}^{\eta}}\left|\Gamma\left(s,\overline{X}_{n_{k}}^{\eta}(s),U_{n_{k}}^{\eta}(s)\right)-\Gamma\left(s,X^{\eta}(s),U_{n_{k}}^{\eta}(s)\right)\right|ds\right)&=o_{1/k}(1),\\
\mathbb{E}\left(\int_{t}^{\tau^{\eta}\wedge\overline{\tau}_{\delta,n_{k}}^{\eta}}\left|\Gamma\left(s,X^{\eta}(s),U_{n_{k}}^{\eta}(s)\right)-\Gamma\left(s,X^{\eta}(s),U^{\eta}(s)\right)\right|ds\right)&=o_{1/k}(1).
\end{align*}
Moreover, using a similar argument as in the proof of Lemma \ref{lem:ProbSndelta}, if we set
\begin{align*}
\overline{\mathcal{S}}^{\eta}_{\delta,k}:=\left\{\omega\in\Omega:\sup_{\ell\in[t,T]}\left|X^{\eta}\left(\ell\wedge\left(\tau^{\eta}\vee\overline{\tau}_{\delta,n_{k}}^{\eta}\right)\right)(\omega)-\overline{X}^{\eta}_{n_{k}}\left(\ell\wedge\left(\tau^{\eta}\vee\overline{\tau}_{\delta,n_{k}}^{\eta}\right)\right)(\omega)\right|>\frac{\delta}{2}\right\},
\end{align*}
then
\begin{align}\label{eq:tautaudeltanketa}
\tau^{\eta}(\omega)\wedge\overline{\tau}_{\delta,n_{k}}^{\eta}(\omega)=\tau^{\eta}(\omega),\quad\text{for all }\,\omega\in\left(\overline{\mathcal{S}}^{\eta}_{\delta,k}\right)^{c},
\end{align}
and by \eqref{eq:ApproxSDESolsFinControl} and Chebyshev's inequality,
\begin{align}\label{eq:ProbOverlineSetadeltak}
\mathbb{P}\left(\overline{\mathcal{S}}^{\eta}_{\delta,k}\right)\rightarrow 0,\quad\text{as }\,k\rightarrow\infty.
\end{align}
Hence,
\begin{align*}
\mathbb{E}\left(\int_{\tau^{\eta}\wedge\overline{\tau}_{\delta,n_{k}}^{\eta}}^{\tau^{\eta}}\left|\Gamma\left(s,X^{\eta}(s),U^{\eta}(s)\right)\right|ds\right)&=\mathbb{E}\left(\int_{\tau^{\eta}\wedge\overline{\tau}_{\delta,n_{k}}^{\eta}}^{\tau^{\eta}}\left|\Gamma\left(s,X^{\eta}(s),U^{\eta}(s)\right)\right|{\bf 1}_{\overline{\mathcal{S}}^{\eta}_{\delta,k}}ds\right)\\
&\leq T\left\|\Gamma\right\|_{L^{\infty}(\overline{Q^{0}}\times\mathcal{U})}\mathbb{P}\left(\overline{\mathcal{S}}^{\eta}_{\delta,k}\right)\rightarrow 0,\quad\text{as }\,k\rightarrow\infty.
\end{align*}
Therefore, we obtain that, as $k\rightarrow\infty$,
\begin{align}\label{eq:EstDiffIntGammaXnketaXeta}
\mathbb{E}\left(\left|\int_{t}^{\tau^{\eta}\wedge\overline{\tau}_{\delta,n_{k}}^{\eta}}\Gamma\left(s,\overline{X}_{n_{k}}^{\eta}(s),U_{n_{k}}^{\eta}(s)\right)ds-\int_{t}^{\tau^{\eta}}\Gamma\left(s,X^{\eta}(s),U^{\eta}(s)\right)ds\right|\right)=o_{1/k}(1).
\end{align}

Next, by Lemma \ref{lem:ConvOverlineWdeltanWdelta},
\begin{align*}
&\mathbb{E}\left(\left|\overline{W}_{\delta,n_{k}}\left(\tau^{\eta}\wedge\overline{\tau}_{\delta,n_{k}}^{\eta},\overline{X}_{n_{k}}^{\eta}\left(\tau^{\eta}\wedge\overline{\tau}_{\delta,n_{k}}^{\eta}\right)\right)-\widehat{W}_{\delta}\left(\tau^{\eta},X^{\eta}\left(\tau^{\eta}\right)\right)\right|\right)\\
&\quad\leq\mathbb{E}\left(\left|\overline{W}_{\delta,n_{k}}\left(\tau^{\eta}\wedge\overline{\tau}_{\delta,n_{k}}^{\eta},\overline{X}_{n_{k}}^{\eta}\left(\tau^{\eta}\wedge\overline{\tau}_{\delta,n_{k}}^{\eta}\right)\right)-\widehat{W}_{\delta}\left(\tau^{\eta}\wedge\overline{\tau}_{\delta,n_{k}}^{\eta},\overline{X}_{n_{k}}^{\eta}\left(\tau^{\eta}\wedge\overline{\tau}_{\delta,n_{k}}^{\eta}\right)\right)\right|\right)\\
&\qquad\,+\mathbb{E}\left(\left|\widehat{W}_{\delta}\left(\tau^{\eta}\wedge\overline{\tau}_{\delta,n_{k}}^{\eta},\overline{X}_{n_{k}}^{\eta}\left(\tau^{\eta}\wedge\overline{\tau}_{\delta,n_{k}}^{\eta}\right)\right)-\widehat{W}_{\delta}\left(\tau^{\eta},X^{\eta}\left(\tau^{\eta}\right)\right)\right|\right)\\
&\quad\leq\left\|\overline{W}_{\delta,n_k}-\widehat{W}_{\delta}\right\|_{L^{\infty}(\overline{Q^0})}+\mathbb{E}\left(\left|\widehat{W}_{\delta}\left(\tau^{\eta}\wedge\overline{\tau}_{\delta,n_{k}}^{\eta},\overline{X}_{n_{k}}^{\eta}\left(\tau^{\eta}\wedge\overline{\tau}_{\delta,n_{k}}^{\eta}\right)\right)-\widehat{W}_{\delta}\left(\tau^{\eta},X^{\eta}\left(\tau^{\eta}\right)\right)\right|\right)\\
&\quad =o_{1/k}(1)+\mathbb{E}\left(\left|\widehat{W}_{\delta}\left(\tau^{\eta}\wedge\overline{\tau}_{\delta,n_{k}}^{\eta},\overline{X}_{n_{k}}^{\eta}\left(\tau^{\eta}\wedge\overline{\tau}_{\delta,n_{k}}^{\eta}\right)\right)-\widehat{W}_{\delta}\left(\tau^{\eta},X^{\eta}\left(\tau^{\eta}\right)\right)\right|\right),\quad k\rightarrow\infty.
\end{align*}
Moreover, by \eqref{eq:ApproxSDESolsFinControl}, \eqref{eq:tautaudeltanketa}, \eqref{eq:ProbOverlineSetadeltak}, and the uniform continuity of $\widehat{W}_{\delta}$ on $\overline{{Q}^{0}}$,
\begin{align*}
&\mathbb{E}\left(\left|\widehat{W}_{\delta}\left(\tau^{\eta}\wedge\overline{\tau}_{\delta,n_{k}}^{\eta},\overline{X}_{n_{k}}^{\eta}\left(\tau^{\eta}\wedge\overline{\tau}_{\delta,n_{k}}^{\eta}\right)\right)-\widehat{W}_{\delta}\left(\tau^{\eta},X^{\eta}\left(\tau^{\eta}\right)\right)\right|\right)\\
&\quad =\mathbb{E}\left(\left|\widehat{W}_{\delta}\left(\tau^{\eta}\wedge\overline{\tau}_{\delta,n_{k}}^{\eta},\overline{X}_{n_{k}}^{\eta}\left(\tau^{\eta}\wedge\overline{\tau}_{\delta,n_{k}}^{\eta}\right)\right)-\widehat{W}_{\delta}\left(\tau^{\eta},X^{\eta}\left(\tau^{\eta}\right)\right)\right|{\bf 1}_{\overline{\mathcal{S}}_{\delta,k}^{\eta}}\right)\\
&\qquad\,+\mathbb{E}\left(\left|\widehat{W}_{\delta}\left(\tau^{\eta}\wedge\overline{\tau}_{\delta,n_{k}}^{\eta},\overline{X}_{n_{k}}^{\eta}\left(\tau^{\eta}\wedge\overline{\tau}_{\delta,n_{k}}^{\eta}\right)\right)-\widehat{W}_{\delta}\left(\tau^{\eta},X^{\eta}\left(\tau^{\eta}\right)\right)\right|{\bf 1}_{(\overline{\mathcal{S}}_{\delta,k}^{\eta})^{c}}\right)\\
&\quad\leq 2\left\|\widehat{W}_{\delta}\right\|_{L^{\infty}(\overline{Q^{0}})}\mathbb{P}\left(\overline{\mathcal{S}}_{\delta,k}^{\eta}\right)+\mathbb{E}\left(\left|\widehat{W}_{\delta}\left(\tau^{\eta},\overline{X}_{n_{k}}^{\eta}\left(\tau^{\eta}\right)\right)-\widehat{W}_{\delta}\left(\tau^{\eta},X^{\eta}\left(\tau^{\eta}\right)\right)\right|\right)\\
&\quad\leq o_{1/k}(1)+\mathbb{E}\left(\varpi_{\delta}\left(\left|\overline{X}_{n_{k}}^{\eta}\left(\tau^{\eta}\right)-X^{\eta}\left(\tau^{\eta}\right)\right|\right)\right)\\
&\quad\leq o_{1/k}(1)+\varpi_{\delta}\left(\mathbb{E}\left(\left|\overline{X}_{n_{k}}^{\eta}\left(\tau^{\eta}\right)-X^{\eta}\left(\tau^{\eta}\right)\right|\right)\right)=o_{1/k}(1),\quad k\rightarrow\infty,
\end{align*}
where $\varpi_{\delta}$ is a concave modulus of continuity of $\widehat{W}_{\delta}$ in $\overline{Q^{0}}$. Therefore, we obtain
\begin{align}\label{eq:EstDiffOverlineWdeltakWdelta}
\mathbb{E}\left(\left|\overline{W}_{\delta,n_{k}}\left(\tau^{\eta}\wedge\overline{\tau}_{\delta,n_{k}}^{\eta},\overline{X}_{n_{k}}^{\eta}\left(\tau^{\eta}\wedge\overline{\tau}_{\delta,n_{k}}^{\eta}\right)\right)-\widehat{W}_{\delta}\left(\tau^{\eta},X^{\eta}\left(\tau^{\eta}\right)\right)\right|\right)=o_{1/k}(1),\quad k\rightarrow\infty.
\end{align}

Combining Lemma \ref{lem:ConvOverlineWdeltanWdelta}, \eqref{eq:Approxeta}, \eqref{eq:EstDiffIntGammaXnketaXeta}, and \eqref{eq:EstDiffOverlineWdeltakWdelta}, and taking limits, as $k\rightarrow\infty$, in both sides of \eqref{eq:OverlineWdeltankleq}, we thus have
\begin{align*}
\widehat{W}_{\delta}(t,x)&\leq\mathbb{E}\left(\int_{t}^{\tau^{\eta}}\Gamma\left(s,X^{\eta}(s),U^{\eta}(s)\right)ds+\widehat{W}_{\delta}\left(\tau^{\eta},X^{\eta}\left(\tau^{\eta}\right)\right)\right)\\
&\leq\inf_{U\in\mathcal{A}_{\mu}}\mathbb{E}\left(\int_{t}^{\tau}\Gamma\left(s,X(s),U(s)\right)ds+\widehat{W}_{\delta}\left(\tau,X(\tau)\right)\right)+\eta.
\end{align*}
Since $\eta>0$ is arbitrary, this implies
\begin{align*}
\widehat{W}_{\delta}(t,x)\leq\inf_{U\in\mathcal{A}_{\mu}}\mathbb{E}\left(\int_{t}^{\tau}\Gamma\left(s,X(s),U(s)\right)ds+\widehat{W}_{\delta}\left(\tau,X(\tau)\right)\right).
\end{align*}
On the other hand, by Remark \ref{rem:ImprovDPFinite},
\begin{align*}
\overline{W}_{\delta,n_{k}}(t,x)&=\inf_{U_{n_{k}}\in\mathcal{A}_{\mu}^{n_{k}}}\left(\int_{t}^{\tau}\Gamma\left(s,\overline{X}_{n_{k}}(s),U_{n_{k}}(s)\right)ds+\overline{W}_{\delta,n_{k}}\left(\tau,\overline{X}_{n_{k}}(\tau)\right)\right)\\
&\geq\inf_{U\in\mathcal{A}_{\mu}}\left(\int_{t}^{\tau}\Gamma\left(s,X(s),U(s)\right)ds+\overline{W}_{\delta,n_{k}}\left(\tau,X(\tau)\right)\right).
\end{align*}
Letting $k\rightarrow\infty$ above and using Lemma \ref{lem:ConvOverlineWdeltanWdelta} provides the reverse inequality, and hence completes the proof of the theorem.
\end{proof}

\begin{lemma}\label{lem:Comparison}
Let Assumptions \ref{assump:SDECoefs}, \ref{assump:GammaPsi}, \ref{assump:Domain O}, and \ref{assump:ellpiticity along boundary} be valid, and let $\Psi\in C^{1+\alpha/2,\,2+\alpha}(\overline{Q^{0}})$ for some $\alpha>0$. Let $W$ be the viscosity solution to \eqref{eq:HJB} with terminal-boundary condition \eqref{eq:TermBoundCondHJB}.
Then
\begin{align*}
\left\|W-\widehat{W}_{\delta}\right\|_{L^{\infty}(\overline{Q^{0}})}=o_{\delta}(1),\quad \text{as }\,\delta\rightarrow 0.
\end{align*}
\end{lemma}
\begin{proof}
In view of Lemma \ref{lem:barfun1}, there exist a uniformly continuous viscosity subsolution $\overline{\psi}_{\delta}$ and a uniformly continuous viscosity supersolution $\overline{\psi}^{\delta}$ to \eqref{eq:HJBDelta1} such that $\overline{\psi}^{\delta}=\overline{\psi}_{\delta}=\Psi$ on $\partial_{{\rm np}} Q_{\delta}$, where the modulus of continuity of $\overline{\psi}^{\delta}$ and $\overline{\psi}_{\delta}$ are independent of $\delta$. Therefore, since $\Psi\in C^{1+\alpha/2,\,2+\alpha}(\overline{Q^{0}})$, we have
\begin{align*}
a_{\delta}:=\sup_{(t,x)\in Q_{\delta}\cap\partial_{\rm np}Q}\left|\widehat{W}_{\delta}(t,x)-W(t,x)\right|=o_{\delta}(1),\quad \text{as }\delta\rightarrow 0.
\end{align*}
Since $\widehat{W}_{\delta}-a_{\delta}$ and $\widehat{W}_{\delta}+a_{\delta}$ are viscosity solutions to \eqref{eq:HJB}, and since
\begin{align*}
\widehat{W}_{\delta}-a_{\delta}\leq W\leq \widehat{W}_{\delta}+a_{\delta}\quad\text{on }\,\partial_{\text{np}}Q,
\end{align*}
the lemma follows immediately from the comparison principle.
\end{proof}

We can now state a representation formula for $W$ with smooth terminal-boundary condition and a general control set.
\begin{theorem}\label{thm:DPVisSmoothTermBound}
Let Assumptions \ref{assump:SDECoefs}, \ref{assump:GammaPsi}, \ref{assump:Domain O}, and \ref{assump:ellpiticity along boundary} be valid, and let $\Psi\in C^{1+\alpha/2,\,2+\alpha}(\overline{Q^{0}})$ for some $\alpha>0$. Let $W$ be the viscosity solution to \eqref{eq:HJB} with terminal-boundary condition \eqref{eq:TermBoundCondHJB}. Let $t\in[0,T]$, let $\mu_{1}=(\Omega,\mathscr{F},\mathscr{F}_{s}^{t},\mathbb{P},\mathcal{W},\widetilde{\mathcal{W}},\mathcal{L})$ be an extended generalized reference probability space, and set $\mu=(\Omega,\mathscr{F},\mathscr{F}_{s}^{t},\mathbb{P},\mathcal{W},\mathcal{L})$. Then, for any $x\in\overline{O}$,
\begin{align}\label{eq:DPmain result}
W(t,x)=\inf_{U\in\mathcal{A}_{\mu}}\mathbb{E}\left(\int_{t}^{\tau(t,x)}\Gamma\left(s,X(s;t,x),U(s)\right)ds+W\left(\tau(t,x),X(\tau(t,x);t,x)\right)\right).
\end{align}
\end{theorem}
\begin{proof}
The result follows by taking $\delta\rightarrow 0$ in \eqref{eq:DPWdelta} and using Lemma \ref{lem:Comparison}.
\end{proof}

Finally, we show that $\Psi\in C^{1+\alpha/2,\,2+\alpha}(\overline{Q^{0}})$ is not needed for establishing the representation formula for $W$. In fact, we only need Assumption \ref{assump:GammaPsi}-(i).
\begin{theorem}\label{thm:DPVisContTermBound}
Let Assumptions \ref{assump:SDECoefs}, \ref{assump:GammaPsi}, \ref{assump:Domain O}, and \ref{assump:ellpiticity along boundary} be valid. Let $W$ be the viscosity solution to \eqref{eq:HJB} with terminal-boundary condition \eqref{eq:TermBoundCondHJB}. Let $t\in[0,T]$, let $\mu_{1}=(\Omega,\mathscr{F},\mathscr{F}_{s}^{t},\mathbb{P},\mathcal{W},\widetilde{\mathcal{W}},\mathcal{L})$ be an extended generalized reference probability space, and set $\mu=(\Omega,\mathscr{F},\mathscr{F}_{s}^{t},\mathbb{P},\mathcal{W},\mathcal{L})$. Then, \eqref{eq:DPmain result} holds for any $x\in\overline{O}$.
\end{theorem}
\begin{proof}
For each $n\in\mathbb{N}$, let $\Psi_{n}\in C^{1+\alpha/2,\,2+\alpha}(\overline{Q^{0}})$, for some small $\alpha>0$, be such that $\Psi_{n}\rightarrow\Psi$ uniformly in $\overline{Q^{0}}$ as $n\rightarrow\infty$. Also, for each $n\in\mathbb{N}$, let $W_{n}$ be the viscosity solution to \eqref{eq:HJB} with $W=\Psi_{n}$ on $\partial_{\text{np}}Q$. By Theorem \ref{thm:DPVisSmoothTermBound},
\begin{align*}
W_{n}(t,x)=\inf_{U\in\mathcal{A}_{\mu}}\mathbb{E}\left(\int_{t}^{\tau(t,x)}\Gamma\left(s,X(s;t,x),U(s)\right)ds+W_{n}\left(\tau
(t,x),X(\tau(t,x);t,x)\right)\right).
\end{align*}
A comparison argument like the one used to prove Lemma \ref{lem:Comparison} ensures that $W_{n}\rightarrow W$ uniformly in $\overline{Q^{0}}$ as $n\rightarrow\infty$. Taking limits on both sides of the above quality, as $n\rightarrow\infty$, completes the proof.
\end{proof}

We conclude this section with a remark and a corollary.

\begin{remark}\label{rem:ImprovDPFgeneral}
Straightforward modifications of arguments of this section also establish the following version of the Dynamic Programming Principle. If the assumptions of Theorem \ref{thm:DPVisContTermBound} are satisfied, $t\in[0,T]$ and a generalized reference probability space $\mu$ is as in Theorem \ref{thm:DPVisContTermBound}, then for any $x\in\overline{O}$
\begin{equation*}
W(t,x)=\inf_{(U,\theta_{U})\in\widetilde{\mathcal{A}}_{\mu}}\mathbb{E}\left(\int_{t}^{\theta_{U}\wedge\tau}\Gamma\left(s,X(s;t,x),U(s)\right)ds+W\left(\theta_{U}\wedge\tau,X(\theta_{U}\wedge\tau;t,x)\right)\right).
\end{equation*}
\end{remark}
\begin{corollary}\label{cor:final}
Under the assumptions of Theorem \ref{thm:DPVisContTermBound}, for any $(t,x)\in \overline{Q}$,
\begin{equation*}
W(t,x)=\inf_{(U,\theta_{U})\in\mathcal{A}_{t}}\mathbb{E}\left(\int_{t}^{\theta_{U}\wedge\tau}\Gamma\left(s,X(s;t,x),U(s)\right)ds+W\left(\theta_{U}\wedge\tau,X(\theta_{U}\wedge\tau;t,x)\right)\right).
\end{equation*}
In particular, taking $\theta_{U}=T$ for every $U$, we obtain that, for any $(t,x)\in \overline{Q}$,
\begin{equation*}
W(t,x)=\inf_{U\in\mathcal{A}_{t}}\mathbb{E}\left(\int_{t}^{\tau}\Gamma\left(s,X(s;t,x),U(s)\right)ds+W\left(\tau,X(\tau;t,x)\right)\right).
\end{equation*}
\end{corollary}
\begin{proof}
The corollary follows from Remarks \ref{rem:EquiProbSpaces} and \ref{rem:ImprovDPFgeneral}.
\end{proof}

\section{Construction of Viscosity Sub/Supersolutions}\label{sec:VisConstrn}

In this section, we construct continuous sub/supersolutions to various equations. We will only discuss the case of equation \eqref{eq:HJB} with all details since the construction for other equations is the same as they satisfy the same uniform conditions. The construction of sub/supersolutions is very similar, and essentially is the same as that in \cite{Mou1:2016} in many respects. We present it here for the sake of completeness.

We begin with a preliminary lemma for which we need the following assumption.
\begin{assumption}\label{assump:Domain O exterior}
\begin{itemize}
\item [(i)] $\mathscr{O}\subset\mathbb{R}^{d}$ is a bounded domain which satisfies the uniform exterior ball condition with a uniform radius $r_{\mathscr{O}}>0$. That is, for any $x\in\partial \mathscr{O}$, there exists $y_{x}\in \mathscr{O}^{c}$, such that $\overline{B}_{r_{\mathscr{O}}}(y_{x})\cap\overline{\mathscr{O}}=\{x\}$.
\item [(ii)] There exists a constant $\lambda_{\mathscr{O}}>0$ such that, for any $x\in\partial \mathscr{O}$, $t\in[0,T]$, and any $u\in\mathcal{U}$,
    \begin{align*}
    \frac{(y_{x}-x)}{|y_{x}-x|}\sigma(t,x,u)\sigma^{T}(t,x,u)\frac{(y_{x}-x)^{T}}{|y_{x}-x|}\geq \lambda_{\mathscr{O}}.
    \end{align*}
\end{itemize}
\end{assumption}

By Theorem \ref{thm:proxdomain}, under Assumptions \ref{assump:Domain O} and \ref{assump:ellpiticity along boundary}, $\mathscr{O}=O_{\delta}$ satisfies Assumption \ref{assump:Domain O exterior} with some $r_{\mathscr{O}}$ and $\lambda_{\mathscr{O}}$ independent of $\delta$. Also $\mathscr{O}=O$ satisfies Assumption \ref{assump:Domain O exterior}.

\begin{lemma}\label{lem:barfunbdy}
Let $\mathscr{O}$ be a bounded domain with smooth boundary $\partial\mathscr{O}$. Let $\Gamma:\overline{Q^{0}}\times\mathcal{U}\rightarrow\mathbb{R}$, $b:\overline{Q^{0}}\times\mathcal{U}\rightarrow\mathbb{R}^{d}$, and $\sigma:\overline{Q^{0}}\times\mathcal{U}\rightarrow\mathbb{R}^{d\times m_{1}}$ be bounded, and let $\gamma:\overline{Q^{0}}\times\mathcal{U}\times\mathbb{R}^{m_{2}}\rightarrow\mathbb{R}^{d}$ be $\mathcal{B}(\mathbb{R}^{m_{2}})$-measurable with respect to $z$ and satisfy \eqref{eq:inftybounds}. Let Assumption \ref{assump:Domain O exterior} be valid. Then, there exist $\delta_{0}\in(0,1)$, $\kappa>0$, and a Lipschitz function $\psi:\mathbb{R}^{d}\rightarrow [0,\infty)$ satisfying
\begin{align*}
\psi=0\,\,\,\text{on }\,\mathscr{O}^{c}\,;\quad\psi\geq\kappa\,\,\,\text{on }\,\overline{\mathscr{O}}_{-\delta_{0}};\quad\psi\in C^{2}(\mathscr{O}\setminus\overline{\mathscr{O}}_{-\delta_{0}}),
\end{align*}
where $\mathscr{O}_{-\delta_{0}}:=\{x\in\mathscr{O}:\,{\rm dist}(x,\partial\mathscr{O})>\delta_{0}\}$, such that for any $(t,x)\in(0,T)\times(\mathscr{O}\setminus\overline{\mathscr{O}}_{-\delta_{0}})$ and $u\in\mathcal{U}$,
\begin{align*}
\mathscr{A}^{u}\psi(t,x)\leq -\kappa,
\end{align*}
where the generator $\mathscr{A}^{u}$ is given by \eqref{eq:GenX}.
\end{lemma}
\begin{proof}
Let $d_{\mathscr{O}}(x)={\rm dist}(x,\mathscr{O}^{c})$, $x\in\mathbb{R}^{d}$. Since $\mathscr{O}$ has a smooth boundary, let $\delta_{1}>0$ be such that $d_{\mathscr{O}}(\cdot)\in C^{2}(\mathscr{D}_{2\delta_{1}})$, where, for any $r>0$, $\mathscr{D}_{r}:=\{x\in\overline{\mathscr{O}}:\,d_{\mathscr{O}}(x)<r\}$. Let
\begin{align*}
\beta(t):=\int_{\{C\rho(z)\geq t\}}\rho(z)\,\nu(dz),\quad\Theta(t):=\int_{0}^{t}2\exp\left(-Ls-L\int_{0}^{s}\beta(\theta)\,d\theta\right)ds-t,\quad t\geq 0,
\end{align*}
where $C>0$ is given in \eqref{eq:inftybounds}, and where $L>0$ will be determined later. Clearly, from the above construction, there exists a constant $t_{0}=t_{0}(C,L;\rho)>0$, depending on the constants $C$ and $L$ as well as the function $\rho$, such that for any $t\in(0,t_{0})$, $\Theta'(t)\geq 1/2$. Letting $\delta_{2}:=\min(t_{0},\delta_{1})/2$, we define $\psi:\mathbb{R}^{d}\rightarrow\mathbb{R}$ via
\begin{align*}
\psi(x):=\left\{\begin{array}{ll} \Theta(d_{\mathscr{O}}(x)),\quad &\text{if }\,d_{\mathscr{O}}(x)<\delta_{2}, \\ \Theta(\delta_{2}),\quad &\text{if }\,d_{\mathscr{O}}(x)\geq\delta_{2}, \end{array}\right.
\end{align*}
and set $\delta_{0}<\delta_{2}/4$ to be chosen later. It is easy to see that $\psi=0$ on $\mathscr{O}^{c}$, and that $\psi$ is a Lipschitz function on $\mathbb{R}^{d}$ with Lipschitz constant $1$. Moreover, it follows from the above definition of $\psi$ that there exists a constant $\kappa\in(0,1)$, such that $\psi\geq\kappa$ on $\mathscr{O}_{-\delta_{0}}$, and that $\psi\in C^{2}(\mathscr{O}\setminus\overline{\mathscr{O}}_{-\delta_{0}})$. For any $x\in\mathscr{O}\setminus\overline{\mathscr{O}}_{-\delta_{0}}$, we have
\begin{align*}
D\psi(x)=\Theta'(d_{\mathscr{O}}(x))Dd_{\mathscr{O}}(x),\quad D^{2}\psi(x)=\Theta'(d_{\mathscr{O}}(x))D^{2}d_{\mathscr{O}}(x)+\Theta''(d_{\mathscr{O}}(x))Dd_{\mathscr{O}}(x)\otimes Dd_{\mathscr{O}}(x).
\end{align*}
Since $d_{\mathscr{O}}(\cdot)\in C^{2}(\mathscr{D}_{2\delta_{1}})$, $\|D^{2}d_{\mathscr{O}}\|_{L^{\infty}(\overline{\mathscr{D}}_{\delta_{1}})}<\infty$. In the rest of the proof, we will denote by $K$ any generic constant (the constant $K$ may vary from one expression to another). Notice however that the constant $K$ will not depend $L$. We choose $\delta_{0}>0$ sufficiently small so that, by Assumption \ref{assump:Domain O exterior}-(ii), for any $u\in\mathcal{U}$ and $(t,x)\in(0,T)\times(\mathscr{O}\setminus\overline{\mathscr{O}}_{-\delta_{0}})$,
\begin{align}\label{eq:EstAuDrifVol}
\frac{1}{2}{\rm tr}\left(a(t,x,u)D^{2}\psi(x)\right)\leq K+\frac{\lambda_{\mathscr{O}}}{2}\,\Theta''(d_{\mathscr{O}}(x)),\quad |b(t,x,u)\cdot D\psi(x)|\leq K,
\end{align}
where we used the fact that $\Theta''(t)\leq 0$, for any $t\geq 0$, in the first inequality above. Moreover, for any fixed $u\in\mathcal{U}$ and $(t,x)\in(0,T)\times(\mathscr{O}\setminus\overline{\mathscr{O}}_{-\delta_{0}})$,
\begin{align}
&\int_{\mathbb{R}_{0}^{m_{2}}}\left(\psi(x+\gamma(t,x,u,z))-\psi(x)-D\psi(x)\cdot\gamma(t,x,u,z)\right)\nu(dz)\nonumber\\
&\quad =\int_{|\gamma(t,x,u,z)|\leq d_{\mathscr{O}}(x)}\left(\psi(x+\gamma(t,x,u,z))-\psi(x)-D\psi(x)\cdot\gamma(t,x,u,z)\right)\nu(dz)\nonumber\\
\label{eq:DecomAuInt} &\qquad\,+\int_{|\gamma(t,x,u,z)|>d_{\mathscr{O}}(x)}\left(\psi(x+\gamma(t,x,u,z))-\psi(x)-D\psi(x)\cdot\gamma(t,x,u,z)\right)\nu(dz).
\end{align}
Since $\Theta''(t)\leq 0$ for any $t\geq 0$, using \eqref{eq:inftybounds}, we have
\begin{align}
&\int_{|\gamma(t,x,u,z)|\leq d_{\mathscr{O}}(x)}\left(\psi(x+\gamma(t,x,u,z))-\psi(x)-D\psi(x)\cdot \gamma(t,x,u,z)\right)\nu(dz)\nonumber\\
&\quad =\int_{|\gamma(t,x,u,z)|\leq d_{\mathscr{O}}(x)}\int_{0}^{1}(1-\alpha)\gamma(t,x,u,z)D^{2}\psi(x+\alpha\gamma(t,x,u,z))\gamma^{T}(t,x,u,z)\,d\alpha\,\nu(dz)\nonumber\\
\label{eq:EstAuInt1} &\quad\leq K\int_{|\gamma(t,x,u,z)|\leq d_{\mathscr{O}}(x)}|\gamma(t,x,u,z)|^{2}\nu(dz)\leq K\int_{\mathbb{R}_{0}^{m_{2}}}\rho^{2}(z)\,\nu(dz)\leq K.
\end{align}
Furthermore, since $\psi$ is a Lipschitz function in $\mathbb{R}^{d}$, by \eqref{eq:inftybounds} again, we have
\begin{align}
&\int_{|\gamma(t,x,u,z)|>d_{\mathscr{O}}(x)}\left(\psi(x+\gamma(t,x,u,z))-\psi(x)-D\psi(x)\cdot\gamma(t,x,u,z)\right)\nu(dz)\nonumber\\
&\quad\leq\int_{C\rho(z)>d_{\mathscr{O}}(x)}\left|\psi(x+\gamma(t,x,u,z))-\psi(x)-D\psi(x)\cdot\gamma(t,x,u,z)\right|\nu(dz)\nonumber\\
\label{eq:EstAuInt2} &\quad\leq K\int_{C\rho(z)>d_{\mathscr{O}}(x)}\rho(z)\nu(dz)=K\beta(d_{\mathscr{O}}(x)).
\end{align}
Therefore, by combining \eqref{eq:EstAuDrifVol}-\eqref{eq:EstAuInt2}, for any $u\in\mathcal{U}$ and $(t,x)\in(0,T)\times(\mathscr{O}\setminus\overline{\mathscr{O}}_{\delta_{0}})$, we have
\begin{align*}
\mathscr{A}^{u}\,\psi(t,x)&\leq\frac{\lambda_{\mathscr{O}}}{2}\,\Theta''(d_{\mathscr{O}}(x))+K\left(\beta(d_{\mathscr{O}}(x))+1\right)\\
&\leq -\frac{\lambda_{\mathscr{O}}}{2}\,L\left(\beta(d_{\mathscr{O}}(x))+1\right)\Theta'(d_{\mathscr{O}}(x))+K\left(\beta(d_{\mathscr{O}}(x))+1\right)\\
&\leq -\frac{L}{4}\lambda_{\mathscr{O}}\left(\beta(d_{\mathscr{O}}(x))+1\right)+K\left(\beta(d_{\mathscr{O}}(x))+1\right)\\
&\leq -\beta(d_{\mathscr{O}}(x))-1\leq -1<-\kappa,
\end{align*}
where we set $L=4(K+1)/\lambda_{\mathscr{O}}$.
\end{proof}

\begin{lemma}\label{lem:barfun}
Let $\Psi\in C^{1+\alpha/2,\,2+\alpha}(\overline{Q^{0}})$, and let Assumptions \ref{assump:SDECoefs},
\ref{assump:GammaPsi}, \ref{assump:Domain O}, and \ref{assump:ellpiticity along boundary} be valid. Then, there exist a uniformly continuous viscosity subsolution $\underline{\psi}$ and a uniformly continuous viscosity supersolution $\overline{\psi}$ to \eqref{eq:HJB} such that, $\underline{\psi}=\overline{\psi}=\Psi$ on $\partial_{{\rm np}}Q$, and such that the modulus of continuity of $\underline{\psi}$ and $\overline{\psi}$ only depends on various absolute constants, $\eta,\lambda$ and $\Psi$.
\end{lemma}
\begin{proof}
We first consider the case $\Psi=0$. We extend our non-local parabolic equation by $u_t+\lambda\Delta u=0$ on $[0,T)\times O^c$. We note that
by Assumptions \ref{assump:Domain O} and \ref{assump:ellpiticity along boundary}, $\mathscr{O}=O$ satisfies Assumption
\ref{assump:Domain O exterior} with $r_{\mathscr{O}}=\eta$ and $\lambda_{\mathscr{O}}=\lambda$. Now, by the boundedness of $O$, there exists a sufficiently large constant $R_0$ such that, for any $x\in\partial O$, we have $O\subset B_{R_{0}-1}( y_{x})\setminus B_{\eta}(y_{x})$. By Lemma \ref{lem:barfunbdy}, applied in $B_{R_{0}}(y_{x})\setminus\overline{B}_{\eta}(y_{x})$, there are
$\delta_0>0$, $\kappa>0$, and a non-negative Lipschitz function $\psi_{x}$ on $\mathbb{R}^{d}$ with Lipschitz constant $1$, such that
$\psi_{x}=0$ on $B_{R_{0}}^{c}(y_{x})\cup\overline{B}_{\eta}(y_{x})$, $\psi_{x}\geq\kappa$ on $O\setminus B_{\eta+\delta_{0}}(y_{x})$,
$\psi_{x}\in C^{2}(\overline{O}\cap B_{\eta+\delta_{0}}(y_{x}))$ and, for any $(s,y)\in [0,T)\times(O\cap B_{\eta+\delta_{0}}(y_{x}))$,
\begin{eqnarray*}
\inf_{u\in\mathcal{U}}(\mathscr{A}^{u}\psi_x(s,y)+\Gamma(s,y,u))\leq-\kappa.
\end{eqnarray*}
It follows from the construction that the constants $\delta_0$, $\kappa$ are independent of $x\in\partial O$.

We take a sufficiently large constant $K_{5}>1$ such that $K_{5}\kappa\geq T(\|\Gamma\|_{L^{\infty}(\overline{Q^{0}}\times\mathcal{U})}+1)$.
It follows from the construction of the function $\psi_{x}$ that $K_{5}\psi_{x}$ is a viscosity supersolution to \eqref{eq:HJB} in $[0,T)\times(O\cap B_{\eta+\delta_{0}}(y_{x}))$ and $(\|\Gamma\|_{L^{\infty}(\overline{Q^{0}}\times\mathcal{U})}+1)(T-s)$ is a viscosity supersolution to \eqref{eq:HJB} in $Q$.

We define $\widetilde{\psi}_{x}(s,y):=\min\{(\|\Gamma\|_{L^{\infty}(\overline{Q^{0}}\times\mathcal{U})}+1)(T-s),\,K_{5}\psi_{x}(y)\}$. Then, $\widetilde{\psi}_{x}(s,x)=0$ for any $s\in[0,T)$, $\widetilde{\psi}_{x}(T,y)=0$ for any $y\in\mathbb{R}^{d}$, $\widetilde{\psi}_{x}\geq 0$ on $Q^{0}$ and
\begin{equation*}
\sup_{x\in\partial O}\left(\left\|D_{y}\widetilde{\psi}_{x}\right\|_{L^{\infty}(\overline{Q^{0}})}+\left\|\frac{\partial\widetilde{\psi}_{x}}{\partial s}\right\|_{L^{\infty}(\overline{Q^{0}})}\right)<+\infty.
\end{equation*}
It is easy to see that $\widetilde{\psi}_{x}$ is a viscosity supersolution to \eqref{eq:HJB} in $Q$.

We define $\widetilde{\psi}(s,y):=\inf_{x\in\partial O}\widetilde{\psi}_{x}(s,y)$. Then $\widetilde{\psi}$ is a non-negative viscosity supersolution to \eqref{eq:HJB} in $Q$, $\widetilde{\psi}(s,y)=0$ for any $(s,y)\in[0,T)\times\partial O$, and $\widetilde{\psi}(T,y)=0$ for any $y\in\mathbb{R}^{d}$. Therefore,
\begin{equation*}
\psi(s,y):=\left\{\begin{array}{ll} \widetilde{\psi}(s,y),\quad &\text{if }\,(s,y)\in Q, \\ 0,\quad &\text{if }\,(s,y)\in\partial_{{\rm np}}Q \end{array}\right.
\end{equation*}
is a viscosity supersolution of \eqref{eq:HJB} in $Q$ and $\psi=0$ on $\partial_{{\text np}}Q$.

We now consider the general case when $\Psi$ is an arbitrary $C^{1+\alpha/2,\,2+\alpha}(\overline{Q^{0}})$ function. Consider the following HJB equation
\begin{align}\label{eq:HJB33}
\inf_{u\in\mathcal{U}}(\mathscr{A}^{u}V(s,y)+\widetilde{\Gamma}(s,y,u))=0\quad\text{in }\,Q,
\end{align}
with terminal-boundary condition
\begin{align}\label{eq:TermBoundCondHJB33}
V(s,y)=0\quad\text{on }\,\partial_{{\rm np}}Q,
\end{align}
where
\begin{align*}
\widetilde{\Gamma}(s,y,u):=\Gamma(s,y,u)+\mathscr{A}^{u}\Psi(s,y).
\end{align*}
Since $\Psi\in C^{1+\alpha/2,\,2+\alpha}(\overline{Q^{0}})$, it follows that $\widetilde{\Gamma}:Q^{0}\times \mathcal{U}\rightarrow\mathbb{R}$ is bounded. By the first part of the proof, we know that there is a supersolution $\psi$ to \eqref{eq:HJB33} with terminal-boundary condition \eqref{eq:TermBoundCondHJB33}. We now define $\overline{\psi}:=\psi+\Psi$. Then $\overline{\psi}$ is a viscosity supersolution to \eqref{eq:HJB} with $\overline{\psi}=\Psi$ on $\partial_{{\rm np}}Q$.

Similarly, we can construct a viscosity subsolution to \eqref{eq:HJB} with $\underline{\psi}=\Psi$ on $\partial_{\rm np}Q$.
\end{proof}

\begin{lemma}\label{lem:barfun1}
Let $\Psi\in C^{1+\alpha/2,\,2+\alpha}(\overline{Q^{0}})$. Let Assumptions \ref{assump:SDECoefs}, \ref{assump:GammaPsi}, \ref{assump:Domain O}, and \ref{assump:ellpiticity along boundary} be valid. Let $\{b_{n}\}_{n\in\mathbb{N}}$, $\{\sigma_{n}\}_{n\in\mathbb{N}}$, $\{\gamma_{n}\}_{n\in\mathbb{N}}$, and $\{\Gamma_n\}_{n\in\mathbb{N}}$ be the sequences satisfying Assumption \ref{assump:CondsApproxFunts}. Then, there exists $\delta_{4}>0$ such that, for any $\delta\in(0,\delta_{4})$, there exists a uniformly continuous viscosity subsolution $\psi_{\delta}$ and a uniformly continuous viscosity supersolution $\psi^{\delta}$ to \eqref{eq:HJBDeltan} and to \eqref{eq:HJBDelta}, such that $\psi^{\delta}=\psi_{\delta}=\Psi$ on $\partial_{{\rm np}}Q_{\delta}$. Moreover, for any $\delta\in(0,\delta_{4})$, there exists a uniformly continuous viscosity subsolution $\overline{\psi}_{\delta}$ and a uniformly continuous viscosity supersolution $\overline{\psi}^{\delta}$ to \eqref{eq:HJBTildeDeltan} and to \eqref{eq:HJBDelta1}, such that $\overline{\psi}^{\delta}=\overline{\psi}_{\delta}=\Psi$ on $\partial_{{\rm np}} Q_{\delta}$. The modulus of continuity of $\psi^{\delta}$, $\psi_{\delta}$, $\overline{\psi}^{\delta}$, and $\overline{\psi}_{\delta}$ only depend on various absolute constants, $\eta$, $\lambda$, and $\Psi$ (and are independent of the parameters $n$ and $\delta$ there).
\end{lemma}
\begin{proof}
We note that Assumptions \ref{assump:Domain O}, \ref{assump:ellpiticity along boundary} and \ref{assump:CondsApproxFunts} imply that $\mathscr{O}=O_{\delta}$ satisfies Assumption \ref{assump:Domain O exterior} for sufficiently small $\delta>0$, and with $r_{\mathscr{O}}=\eta/2$ and $\lambda_{\mathscr{O}}=\lambda/2$. We then construct the functions $\psi^{\delta}$, $\psi_{\delta}$, $\overline{\psi}^{\delta}$, and $\overline{\psi}_{\delta}$ as in the proof of Lemma \ref{lem:barfun}.
\end{proof}

\end{document}